\documentclass[12pt]{amsart}
\pdfoutput=1
\usepackage{graphicx}
\usepackage{subfig}
\usepackage{amssymb}
\usepackage{amsmath}
\usepackage{amsthm}
\usepackage{amscd}
\usepackage{epsfig}
\usepackage{xcolor}
\usepackage{mathrsfs}
\usepackage[utf8]{inputenc}
\usepackage[T1]{fontenc}

\newtheorem{theorem}{Theorem}[section]

\newtheorem{lemma}[theorem]{Lemma}

\newtheorem{proposition}[theorem]{Proposition}

\newtheorem{remark}[theorem]{Remark}
\theoremstyle{definition}
\newtheorem{definition}[theorem]{Definition}
\numberwithin{equation}{section}

\newcommand{\eps}{\varepsilon}

\begin{document}

\title{The growth of eigenfunction extrema on p.c.f. fractals}

\author [Hua Qiu and Haoran Tian]{Hua Qiu and Haoran Tian}

\address{School of Mathematics, Nanjing University,
Nanjing, Jiangsu, 210093, P.R. China} \email{huaqiu@nju.edu.cn}

\address{School of Mathematics, Nanjing University,
Nanjing, Jiangsu, 210093, P.R. China} \email{hrtian@hotmail.com}

\subjclass[2000]{Primary: 28A80}

\keywords {Sierpinski gasket, fractal Laplacian, eigenfunctions, extremum counting function}

\thanks {Hua Qiu was supported by the National Natural Science Foundation of China (Grant No. 12471087 and 12531004).}

\begin{abstract}
This paper studies the growth of local extrema of Laplacian eigenfunctions on post-critically finite (p.c.f.) fractals. We establish the sharp two-sided estimate $\#\mathrm{Extr}(u_\lambda)\asymp\lambda^{d_S/2}$ for the Sierpinski gasket, demonstrating that the complexity of eigenfunctions is governed by the spectral dimension $d_S$.  This behavior stands in sharp contrast to the corresponding growth law on Euclidean $n$-dimensional rectangles or balls. The attainment of the exponent $d_S/2$ reflects the high symmetry of the underlying fractal. Our result reveals a distinct spectral-geometric phenomenon on singular spaces.

\end{abstract}

\maketitle

\section{Introduction}\label{sec1}
The study of Laplacians on fractals and their spectral properties constitutes a cornerstone of analysis on fractals, a field largely pioneered by Kigami \cite{K2,K4}. Unlike their Euclidean counterparts, Laplacians on fractals exhibit a wealth of novel and unexpected phenomena, leading to a rich and distinct spectral theory \cite{FS, Ka1,Ka2,K3,KL, Sh, Sh2, T}. Among these phenomena, the intricate oscillatory behavior of eigenfunctions — such as the distribution and growth of their local extrema — remains a topic of considerable interest.

In the classical setting of a smooth compact Riemannian manifold (or a bounded Euclidean domain with suitable boundary conditions), the celebrated Courant nodal domain theorem \cite{CH} provides a fundamental upper bound on the number of nodal domains of an eigenfunction (where a nodal domain means a maximal connected region on which the eigenfunction does not change sign), which, in a nonrigorous sense, hints at a heuristic upper bound for the number of local extrema. Moreover, the Hörmander-type estimates \cite{Ho} show that the $L^\infty$-norm of an eigenfunction grows at most polynomially with the eigenvalue. These results give a picture of eigenfunctions whose geometric complexity increases in a controlled and predictable manner as the energy (eigenvalue) increases.

The landscape on fractals is quite different. The existence of pre-localized eigenfunctions — eigenfunctions that vanish identically on the boundary along with their normal derivatives — is a hallmark of fractals \cite{K3,K4}. Such eigenfunctions can be highly localized and give rise to  a cascade of new eigenfunctions through a localization process. This structure fundamentally alters the asymptotic distribution of eigenvalues and the qualitative behavior of eigenfunctions. Consequently, classical tools and intuition from elliptic PDEs often fail, necessitating new frameworks for understanding the fine properties of eigenfunctions on fractals. 

\begin{figure}[h]
\begin{center}
\includegraphics[width=5.5cm]{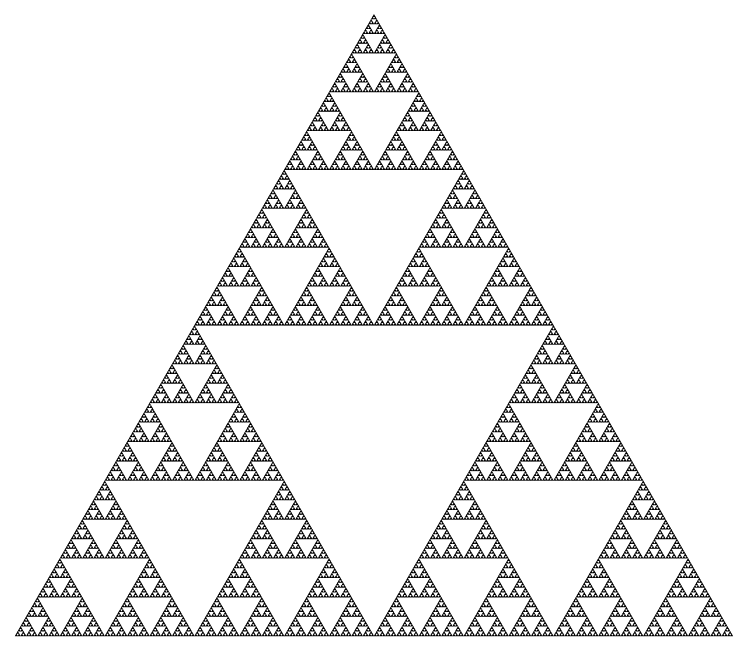}
\caption{The Sierpinski gasket $\mathcal{SG}$.}
\label{figure1}
\end{center}
\end{figure}

This paper is devoted to a study of the oscillatory behavior of eigenfunctions on post-critically finite (p.c.f.) self-similar sets, among which the Sierpinski gasket ($\mathcal{SG}$, see Figure \ref{figure1}) serves as a typical example. In \cite{DSV}, Dalrymple, Strichartz and Vinson performed computational simulations of some eigenfunctions on the $\mathcal{SG}$, illustrating the oscillation of the functions and the distribution of nodal sets, which has inspired this work. In particular, the numerical data in \cite{DSV} suggested a certain pattern for the number of local extrema of eigenfunctions restricted to edges in $\mathcal{SG}$, which was confirmed by the authors in a previous work \cite{QT}.

We introduce and investigate the Dirichlet (or Neumann) extremum counting function 
$\Gamma_D(x)$ (or $\Gamma_N(x)$) for the Laplacian $-\Delta_\mu$, where $\mu$ is the underlying reference measure. For an eigenfunction $u_\lambda$ with eigenvalue $\lambda$, the count 
$\#\mathrm{Extr}(u_\lambda)$ enumerates its distinct local maximum and minimum sets (formally defined in Definition \ref{def1}). The function $\Gamma_D(x)$ (or $\Gamma_N(x)$) then captures the maximum of this number over all Dirichlet (or Neumann) eigenfunctions with eigenvalues up to $x$:
\begin{equation*}
\Gamma_D(x)=\sup\{\#\mathrm{Extr}(u_\lambda): u_\lambda \text{ is a }\lambda \text{-Dirichlet eigenfunction with } 0\leq \lambda\leq x\}.
\end{equation*}
$\Gamma_N(x)$ is defined similarly.

Our primary goal is to establish the asymptotic growth rate of $\Gamma(x)$ (either $\Gamma_D(x)$ or $\Gamma_N(x)$, as they exhibit the same behavior) as $x\to \infty$, linking it directly to the spectral dimension $d_S$ of the fractal, which characterizes a sharp phase transition in the behavior of $\Gamma(x)$.

For the upper bound, we prove in Proposition \ref{prothm1} that under a natural condition (A) — namely, that eigenfunctions with sufficiently small eigenvalues possess at most one extreme set — the growth of $\Gamma(x)$ is at most of the order $x^{d_S/2}$, i.e.
\begin{equation*}
\limsup_{x\to\infty} \frac{\Gamma(x)}{x^{d_S/2}}<\infty.
\end{equation*}
Condition (A) can be interpreted as a form of ``low-energy simplicity''. It is not universal for all fractals, as evidenced by the modified Koch curve \cite{M,Sh2}, where even low-energy eigenfunctions can have infinitely many extrema. Verifying its validity is therefore a key step in the analysis for a given fractal.

Conversely, for the lower bound, Proposition \ref{prothm2} shows that the existence of a pre-localized eigenfunction forces $\Gamma(x)$ to grow at least polynomially. In the lattice case, this lower bound is sharp, matching the upper bound:
\begin{equation*}
\liminf_{x\to\infty}\frac{\Gamma(x)}{x^{d_S/2}}>0.
\end{equation*}
The proof constructs sums of copies of a pre-localized eigenfunction over appropriately chosen cells — an inherently fractal method with no direct analogue in the smooth setting.

Our main result is the following theorem, which gives a precise, uniform, two-sided estimate of $\#\mathrm{Extr}(u_\lambda)$ for eigenfunctions on the Sierpinski gasket $\mathcal{SG}$.
\begin{theorem}\label{thm1}
There exists a constant $C>1$ such that
\begin{equation}\label{equnthm1}
C^{-1}\lambda^{d_S/2}\leq\#\mathrm{Extr}(u_\lambda)\leq C\lambda^{d_S/2}
\end{equation}
holds for any global Dirichlet or Neumann eigenfunction $u_\lambda$ on $\mathcal{SG}$ with  eigenvalue $\lambda$, except for the first non-constant Neumann eigenfunction where $\#\mathrm{Extr}(u_\lambda)=0$.
\end{theorem}
The result demonstrates the highly regular behavior of  $\mathcal{SG}$, where eigenfunction complexity — quantified by the count of local extrema — grows  precisely as the  power law  $d_S/2$. The proof relies on a detailed analysis of the eigenfunction decimation, originally developed by Rammal and Toulouse \cite{RT} and later rigorously established by Shima and Fukushima \cite{Sh, FS}.  

In the classical setting, no rigorous upper bounds are known for the number of critical points, except in the separation-of-variables case \cite[page 10]{Z2}, such as 
$n$-dimensional rectangles or balls—where eigenfunctions may oscillate at high frequencies independently along distinct spatial directions—for which the number of critical points can grow like 
$\lambda^{n/2}$. It is believed \cite[Section 4.2]{Z1} that $\lambda^{n/2}$ is the optimal order, yet even for real analytic Riemannian manifolds, no rigorous results in this
direction are known. On the other hand, no lower bound on the number of critical points exists: Jakobson and Nadirashvili \cite{JN} constructed a Riemannian surface possessing a sequence of eigenfunctions with a fixed finite number of
critical points, thereby negatively answering a question of Yau \cite{Y}, who asked whether the number grows as the eigenvalue increases.

On fractals, as shown  in Theorem \ref{thm1} and Propositions \ref{prothm1} and \ref{prothm2}, the growth is governed by the spectral dimension $d_S$, which takes the role played by the geometric Hausdorff dimension $n$ in the Euclidean setting. The observed order $d_S/2$ for $\mathcal{SG}$ suggests that its high symmetry is a key factor in realizing this maximal growth rate. We therefore conjecture that for a broader class of p.c.f. self-similar sets, the growth order of local extrema is at most $d_S/2$, with this upper bound being attained only in highly symmetric cases  — such as nested fractals, or even beyond the p.c.f. setting, as exemplified by the  Sierpinski carpet — and strictly smaller in the presence of lower symmetry. 

\subsection{Notation and Propositions \ref{prothm1} and \ref{prothm2}}\label{subsec11}
Before ending this section, let us introduce the exact definition of $\#\mathrm{Extr}(u)$ and $\Gamma(x)$ on a p.c.f. self-similar set $K$, which is always assumed to be connected.

Let $V_0$ denote the boundary of $K$, and $\mathscr{D}_\mu$ denote the domain of $\Delta_\mu$. For a function $u\in \mathscr{D}_\mu$ and $p\in V_0$, denote by $(du)_p$ the normal derivative of $u$ at $p$ (see the exact meaning in Section \ref{subsec21}).  
For $\lambda\geq 0$, call a non-trivial  function $u\in\mathscr{D}_\mu$ satisfying $-\Delta_\mu u=\lambda u$ on $K\setminus V_0$ a {\it $\lambda$-eigenfunction} of $-\Delta_\mu$.  An eigenfunction $u$ is said to be a \textit{Dirichlet (or Neumann) eigenfunction} if $u|_{V_0}=0$ (or $du|_{V_0}=0$). In particular, $u$ is called a \textit{pre-localized eigenfunction} if both $u|_{V_0}=0$ and $du|_{V_0}=0$ hold; and a \textit{global eigenfunction} if $\mathrm{supp}\,u=K$.

\begin{definition}\label{def1}

(a). 
Let $u\in\mathscr{D}_\mu$. If there exist non-empty $A\subset K$ and $c\in\mathbb{R}$ such that

(a-1). $A$ is a connected component of $u^{-1}(c)$;

(a-2). $A\cap V_0=\emptyset$;

(a-3). there exists $\delta>0$ such that $u(p)\leq c$ (resp. $u(p)\geq c$) for any $ p\in A_\delta$, where
$A_\delta$ is the $\delta$-neighborhood of $A$,

\noindent then we say that $u$ has a \textit{local maximum ({\rm resp.} minimum)} in $A$ with value $c$. Such a set $A$ will be called a \textit{(local) extreme set} of $u$.

(b). For $u\in\mathscr{D}_\mu$, denote
\begin{equation*}
\mathrm{Extr}(u)=\{A\subset K: A \text{ is an extreme set of } u\text{ for some }c\}.
\end{equation*}
For $x>0$, define
\begin{equation*}
\Gamma_D(x)=\sup\{\#\mathrm{Extr}(u_\lambda): u_\lambda \text { is a $\lambda$-Dirichlet eigenfunction with}\ 0\leq \lambda\leq x\},
\end{equation*}
and call it \textit{the Dirichlet extremum counting function} of $-\Delta_\mu$. \textit{The Neumann extremum counting function} $\Gamma_N(x)$ is defined similarly.
\end{definition}

\noindent{\bf Remark.}
For a $\lambda$-eigenfunction with $\lambda>0$, the values of its local maxima (resp. minima) are positive (resp. negative).  
Indeed, suppose that $u$ has a local maximum in $A$ with value $c\leq 0$. Since $u$ is a $\lambda$-eigenfunction, there exists $\delta>0$ such that $u\leq c\leq 0$ and $\Delta_\mu u=-\lambda u\geq 0$ on $A_\delta$. Moreover, $u$ is non-constant on $A_\delta$. Let $h$ be the harmonic function satisfying $h|_{\partial A_\delta}=u|_{\partial A_\delta}$, and denote by $g(\cdot,\cdot)$ the Green function for $-\Delta_\mu$ on $A_\delta$. It is known that $g(p,q)>0$ for all $p,q\in A_\delta$. Then, for $p\in A$, we have
\begin{equation*}
u(p)=\int_{A_\delta}g(p,q)\lambda u(q)\mathrm{d}\mu(q)+h(p)<h(p)\leq \max u|_{\partial A_{\delta}},
\end{equation*}
which yields a contradiction.

We have certain asymptotic estimates of $\Gamma(x)$ (either $\Gamma_D(x)$ or $\Gamma_N(x)$) as $x\to \infty$. 

\noindent{\bf(A)}: there exists $\lambda_0>0$ such that $\#\mathrm{Extr}(u_\lambda)\leq 1$ for any $\lambda$-eigenfunction $u_\lambda$ with $0\leq\lambda<\lambda_0$.

\begin{proposition}\label{prothm1}
If condition (A) holds, then
\begin{equation*}
\limsup_{x\to\infty}\frac{\Gamma(x)}{x^{d_S/2}}< \infty.
\end{equation*}
\end{proposition}

\begin{proposition}\label{prothm2}
If there exists a pre-localized eigenfunction, then
\begin{equation*}
\liminf_{x\to\infty}\frac{\log \Gamma(x)}{\log x}\geq\kappa
\end{equation*}
for some $0<\kappa\leq d_S/2$.
In particular, for the lattice case (see the exact meaning in Section \ref{sec2}), 
\begin{equation*}
\liminf_{x\to\infty}\frac{\Gamma(x)}{x^{d_S/2}}>0.
\end{equation*}
\end{proposition}

We structure the paper as follows. 

In Section \ref{sec2}, we present the proofs of Propositions \ref{prothm1} and \ref{prothm2}, and provide an equivalent characterization of condition (A). 

Beginning in Section \ref{sec3}, we focus on the canonical Laplacian on the Sierpinski gasket $\mathcal{SG}$. There, we recall the spectral decimation method and state two key preparatory theorems — Theorems \ref{thm4} and \ref{thm5}. Theorem \ref{thm4} confirms the validity of condition (A), while Theorem \ref{thm5} establishes a two-sided estimate of $\#\mathrm{Extr}(u_\lambda)$ for a special class of eigenfunctions.  

The proof of Theorem \ref{thm4} is given in Section \ref{sec4},  followed by the proof of Theorem \ref{thm5} in Section \ref{sec5}. We conclude in Section \ref{sec6} with the proof of Theorem \ref{equnthm1}. 

\section{asymptotic estimate of the extremum counting function}\label{sec2}

The main aim of this section is to prove Propositions \ref{prothm1} and \ref{prothm2}, and to provide an equivalent characterization of condition (A).

Before proceeding, we begin with a brief review of the construction of Laplacians on p.c.f. self-similar sets, and collect some basic facts about the eigenfunctions. All  materials can be found in \cite{K4,St6}.

\subsection{Preliminaries}\label{subsec21}

Let $(X,d)$ be a complete metric space. Denote the set of \textit{symbols} by $S=\{1,2,\cdots,s\}$ with $s\geq2$, and let $\{F_i\}_{i\in S}$ be a collection of contractions on $(X,d)$. We call $\{F_i\}_{i\in S}$ an \textit{iterated function system (i.f.s.)}. Let $K$ be the self-similar set associated with $\{F_i\}_{i\in S}$, i.e. $K$ is the unique non-empty compact set in $X$ satisfying
\[K=\bigcup_{i\in S}F_iK.\]

For $m\geq 1$, we define $W_m=S^m$ as the collection of \textit{words} with \textit{length} $m$, and for $w\in W_m$, write $|w|=m$. For each $w=w_1w_2\cdots w_m\in W_m$, denote
\[F_w=F_{w_1}\circ F_{w_2}\circ\cdots\circ F_{w_m}.\]
By convention, $W_0=\{\varnothing\}$ contains only the empty word, and $F_{\varnothing}={\rm id}$ is the identity map. Set $W_*=\bigcup_{m\geq 0}W_m$.

We define the {\it shift space} $\Sigma=S^\mathbb{N}$ as the collection of all infinite words equipped with the usual product topology. For $\omega=\omega_1\omega_2\cdots\in\Sigma$ and $m\geq 0$, we write $[\omega]_m=\omega_1\omega_2\cdots\omega_m\in W_m$ as the $m$-th \textit{truncation} of $\omega$. Set $[\omega]_0=\varnothing$. For any $w\in W_*$, define the \textit{cylinder set} generated by $w$ as
\[\Sigma_w=\{\omega\in\Sigma:[\omega]_{|w|}=w\}.\]

We say that a finite subset $P\subset W_*$ is a \textit{partition} of $\Sigma$ if $\Sigma_w\cap\Sigma_{w'}=\emptyset$ for any $w\neq w'\in P$ and $\Sigma=\bigcup_{w\in P}\Sigma_w$. Clearly, $W_m$ is a partition for any $m\geq 1$.

Let $\Pi:\Sigma\to K$ be the continuous surjection defined by
\[\{\Pi(\omega)\}=\bigcap_{m\geq 1}F_{[\omega]_m}K.\]
We define the \textit{critical set} $\mathscr{C}\subset\Sigma$ and the \textit{post-critical set} $\mathscr{P}\subset\Sigma$ by
\[\mathscr{C}=\Pi^{-1}\Big(\bigcup_{i\neq j}(F_iK\cap F_jK)\Big),\quad\mathscr{P}=\bigcup_{m\geq 1}\sigma^m(\mathscr{C}),\]
where the \textit{shift map} $\sigma:\Sigma\to\Sigma$ is defined by $\sigma(\omega_1\omega_2\cdots)=\omega_2\omega_3\cdots$. A set $K$ is called a \textit{post-critically finite (p.c.f.)} self-similar set if $\#\mathscr{P}<\infty$. Throughout the paper, we always assume that $K$ is a connected p.c.f. self-similar set.

We denote $V_0=\Pi(\mathscr{P})$ and call it the \textit{boundary} of $K$. Write $L=\# V_0$ and list $V_0=\{p_1,p_2,\cdots,p_L\}$. For $m\geq 1$, set
\[V_m=\bigcup_{w\in W_m}F_wV_0\quad\text{ and }\quad V_*=\bigcup_{m\geq 0}V_m.\]

It is clear that $V_m\subset V_{m+1}$ and $K$ is the closure of $V_*$. For any partition $\Gamma$, we always have $F_wK\cap F_{w'}K=F_wV_0\cap F_{w'}V_0$ for any $w\neq w'\in\Gamma$.

For a finite set $V$, denote $l(V)$ as the collection of all real-valued functions on $V$. Let $H:l(V)\to l(V)$ be a symmetric linear map (matrix), and call it a \textit{(discrete) Laplacian} on $V$ if $H$ is non-positive definite; $Hu=0$ if and only if $u$ is constant on $V$; and $H_{pq}\geq 0$ for any $p\neq q\in V$. Define the associated \textit{(discrete) energy form} $\mathscr{E}_H$ on $V$ by $\mathscr{E}_H(u,v)=-u^\mathrm{t}Hv$ for $u,v\in l(V)$.

Given a Laplacian $H_0$ on $V_0$, and let $\mathbf{r}=(r_1,r_2,\cdots,r_s)$ with $r_i>0,\ 1\leq i\leq s$, define
\[\mathscr{E}_0(u,v):=\mathscr{E}_{H_0}(u,v)\quad\text{ for any } u,v\in l(V_0),\]
and inductively for $m\geq 1$,
\[\mathscr{E}_m(u,v):=\sum_{i\in S}r_i^{-1}\mathscr{E}_{m-1}(u\circ F_i,v\circ F_i)\quad\text{ for any }  u,v\in l(V_m).\]
Let $H_m:=\sum_{w\in W_m}r_w^{-1}R_w^\mathrm{t}H_0R_w$, where $r_w=r_{w_1}r_{w_2}\cdots r_{w_m}$ and $R_w:l(V_m)\to l(V_0)$ is defined by $R_wf=f\circ F_w$ for $w=w_1w_2\cdots w_m\in W_m$. Then $H_m:l(V_m)\to l(V_m)$ is a Laplacian on $V_m$ so that $\mathscr{E}_{H_m}=\mathscr{E}_m$. 

For $m\geq0$ and $u\in  l(V_m)$,  write $\mathscr{E}_m(u):=\mathscr{E}_m(u,u)$ for short.

We call the pair $(H_0,\mathbf{r})$ a \textit{harmonic structure} if for any $v\in l(V_0)$,
\[\mathscr{E}_0(v)=\min\{\mathscr{E}_1(u):u\in l(V_1),\ u|_{V_0}=v\}.\]
Further, $(H_0,\mathbf{r})$ is called {\it regular} if $0<r_i<1$ for all $1\leq i\leq s$. Assume $(H_0,\mathbf{r})$ is a regular harmonic structure, then for each $u\in l(V_*)$, by the self-similarity, the sequence $\{\mathscr{E}_m(u|_{V_m})\}_{m\geq 0}$ is non-decreasing. Define
\[\mathscr{E}(u):=\lim_{m\to\infty}\mathscr{E}_m(u|_{V_m}),\quad\quad\mathscr{F}:=\{u\in l(V_*):\mathscr{E}(u)<\infty\},\]
and call $\mathscr{E}(u)$ the \textit{energy} of $u$. By the regularity of $(H_0,\mathbf{r})$, the function $u\in\mathscr{F}$ can be uniquely extended to a continuous function on $K$, still denoted by $u$. We thus regard $\mathscr{F}$ as a subset of $C(K)$.

Clearly, for $m\geq 0$, $(\mathscr{E},\mathscr{F})$ satisfies the \textit{self-similar identity}
\[\mathscr{E}(u)=\sum_{w\in W_m}r_w^{-1}\mathscr{E}(u\circ F_w)\quad\text{ for any } u\in\mathscr{F}.\]

From now on, we always assume that there exists a regular harmonic structure $(H_0,\mathbf{r})$ on the p.c.f. self-similar set $K$.

For $(\mu_1,\mu_2,\cdots,\mu_s)\in(0,1)^s$ with $\sum_{i\in S}\mu_i=1$, we denote $\mu$ the unique Borel probability measure on $K$ satisfying $\mu=\sum_{i\in S}\mu_i\mu\circ F_i^{-1}$, call it the {\it self-similar measure}  \cite{Hut} associated with $(\mu_1,\mu_2,\cdots,\mu_s)$. For a p.c.f. self-similar set $K$, we have  $\mu(F_wK)=\mu_w:=\mu_{w_1}\mu_{w_2}\cdots\mu_{w_m}$ for any $w=w_1w_2\cdots w_m\in W_m$ with $m\geq 0$.

It is known that $(\mathscr{E},\mathscr{F})$ is a local regular Dirichlet form on $L^2(K,\mu)$. Its infinitesimal generator, the $\mu$-Laplacian $\Delta_\mu$ on $K$, is obtained as a scaled limit of the discrete Laplacians $H_m$ on $V_m$ in the following way.

For $p\in V_m$, let $\psi_{m,p}$ be the unique function in $\mathscr{F}$ that attains the following minimum: $\mathscr{E}(\psi_{m,p})=\min\{\mathscr{E}(u):u\in\mathscr{F},u(p)=1,u(q)=0\text{ for }q\in V_m\setminus\{p\}\}$. For $u\in C(K)$, if there exists $f\in C(K)$ such that
\begin{equation}\label{equation2.1}
\lim_{m\to\infty}\max_{p\in V_m\setminus V_0}|\mu_{m,p}^{-1}(H_mu)(p)-f(p)|=0,
\end{equation}
where $\mu_{m,p}=\int_K\psi_{m,p}\mathrm{d}\mu$, then we say that $u$ is in the domain of the \textit{$\mu$-Laplacian} $\Delta_\mu$ and write $\Delta_\mu u=f$. Denote the domain of $\Delta_\mu$ as $\mathscr{D}_\mu$. By the regularity of $(H_0,\mathbf{r})$ and self-similarity of $\mu$, we have for $u\in \mathscr D_{\mu}$,
\begin{equation}\label{equation2.2}
u\circ F_i\in \mathscr D_\mu,\quad \Delta_\mu(u\circ F_i)=r_i\mu_i(\Delta_\mu u)\circ F_i,\quad\text{ for any }i\in S,
\end{equation}
and by iteration,
\begin{equation*}
u\circ F_w\in \mathscr D_\mu,\quad \Delta_\mu(u\circ F_w)=r_w\mu_w(\Delta_\mu u)\circ F_w,\quad\text{ for any }w\in W_*.
\end{equation*}

It is known that $\mathscr{D}_\mu\subset\mathscr{F}$, and the \textit{Neumann derivative} of $u$ on the boundary, defined by $(\mathrm{d}u)_p=\lim_{m\to\infty}-(H_mu)(p)$, exists for $u\in\mathscr{D}_\mu,\ p\in V_0$. For $u\in\mathscr F$ and $v\in \mathscr D_\mu$, the following \textit{Gauss-Green formula} holds,
\[\mathscr{E}(u,v)=\sum_{p\in V_0}u(p)(\mathrm{d}v)_p-\int_Ku\Delta_\mu v\mathrm{d}\mu.\]

For $u\in\mathscr{D}_\mu$, write
\begin{align*}
u|_{V_0}&=\big(u(p_1),u(p_2),\cdots,u(p_L)\big),\\
\mathrm{d}u|_{V_0}&=\big((\mathrm{d}u)_{p_1},(\mathrm{d}u)_{p_2},\cdots,(\mathrm{d}u)_{p_L}\big),
\end{align*}
and define 
\begin{align*}
\mathscr{D}_{D,\mu}&:=\{u\in\mathscr{D}_\mu: \ u|_{V_0}=\mathbf{0}\},\\
\mathscr{D}_{N,\mu}&:=\{u\in\mathscr{D}_\mu: \ \mathrm{d}u|_{V_0}=\mathbf{0}\}.
\end{align*}

For $\lambda\geq 0$, define
\begin{align*}
&E(\lambda)=\{u\in\mathscr{D}_\mu: -\Delta_\mu u(x)=\lambda u(x) \text{ for any } x\in K\setminus V_0\},\\
&E_D(\lambda)=E(\lambda)\cap\mathscr{D}_{D,\mu}, \quad\text{ and } E_N(\lambda)=E(\lambda)\cap\mathscr{D}_{N,\mu}.
\end{align*}
Note that a non-trivial (not identically zero) function $u\in E(\lambda)$ is a $\lambda$-eigenfunction of $-\Delta_\mu$.  
If $\dim E_D(\lambda)\geq1$, we refer to $\lambda$ as a \textit{Dirichlet eigenvalue} of $-\Delta_\mu$; and if $\dim E_N(\lambda)\geq1$, we refer to $\lambda$ as a \textit{Neumann eigenvalue} of $-\Delta_\mu$. 

\begin{proposition}\label{pro01}
For any $u\in E(\lambda)$ and any $w\in W_*$, we have $u\circ F_w\in E(r_w\mu_w\lambda)$.
\end{proposition}
\begin{proof} This is an immediate consequence of \eqref{equation2.2} and induction.
\end{proof}

By a standard theory, $\dim E_D(\lambda)<\infty$ ($\dim E_N(\lambda)<\infty$) for every $\lambda\geq0$, and  the \textit{Dirichlet (or Neumann) spectrum}, the collection of all Dirichlet (or Neumann) eigenvalues, is discrete and has an only limit point $+\infty$. We list them in an increasing order (each eigenvalue is counted according to its multiplicity) as
\[\Lambda^D:=\{\lambda_1^D,\lambda_2^D,\cdots\}\quad \text{ with }0<\lambda_1^D\leq\lambda_2^D\leq\cdots\]
for the Dirichlet case, and similarly,
\[\Lambda^N:=\{\lambda_1^N,\lambda_2^N,\cdots\}\quad \text{ with }0=\lambda_1^N\leq\lambda_2^N\leq\cdots\]
for the Neumann case.

Denote $\gamma_i=\sqrt{r_i\mu_i}$ for all $i\in S$, and write $\gamma_w=\gamma_{w_1}\gamma_{w_2}\cdots\gamma_{w_m}$ for $w=w_1w_2\cdots w_m\in W_m$.
Let $d_S$ be the unique real number satisfying
\begin{equation}\label{defds}
\sum_{i\in S}\gamma_i^{d_S}=1.
\end{equation}
We call $d_S$ the \textit{spectral dimension} of $-\Delta_\mu$.

For $x>0$, define $\rho_D(x):=\#\{\lambda\leq x: \lambda\in \Lambda^D\}$ and $\rho_N(x):=\#\{\lambda\leq x: \lambda\in \Lambda^N\}$ and call $\rho_D$ (or $\rho_N$) the \textit{Dirichlet (or Neumann) eigenvalue counting function}. In 1993, Kigami and Lapidus \cite{KL} established that
\begin{equation*}
0<\liminf_{x\rightarrow\infty}\rho_*(x)/x^{d_S/2}\leq\limsup_{x\rightarrow\infty}\rho_*(x)/x^{d_S/2}<\infty, \quad \text{ for $*=D,N$.}
\end{equation*}
Moreover,

(1). {\it Non-lattice case}: if $\sum_{i\in S}\mathbb Z \log \gamma_i$ is a dense subgroup of $\mathbb R$, then the limit $\lim_{x\rightarrow\infty}\rho_*(x)/x^{d_S/2}$ exists.

(2). {\it Lattice case}: if $\sum_{i\in S}\mathbb Z \log \gamma_i$ is a discrete subgroup of $\mathbb R$, letting $T>0$ be its generator, then $\rho_*(x)=(G(\log x/2)+o(x))x^{d_S/2}$, where $G$ is a right-continuous $T$-periodic function with $0<\inf G(x)\leq \sup G(x)<\infty$.

\subsection{Proofs of Propositions \ref{prothm1} and \ref{prothm2}}\label{subsec22}

In this subsection, we prove the asymptotic estimate of the extremum counting function $\Gamma(x)$, i.e. Propositions \ref{prothm1} and \ref{prothm2}.

Below is a basic observation about $\#\mathrm{Extr}(u)$.
\begin{lemma}\label{pro1}
Let $u\in\mathscr{D}_\mu$, $k\neq 0$, $w\in W_*$ and let $P$ be a partition of $\Sigma$.

(a). $\#\mathrm{Extr}(u)=\#\mathrm{Extr}(ku)$;

(b). $\sum_{w\in P}\#\mathrm{Extr}(u\circ F_w)\leq\#\mathrm{Extr}(u)\leq \sum_{w\in P}\#\mathrm{Extr}(u\circ F_w)+\#(V_P\setminus V_0),$ where $V_P=\bigcup_{w\in P}F_wV_0$.
\end{lemma}
\begin{proof}
(a) is obvious. (b) follows from the observation that for an extreme set $A$ of $u$, either $A\subset F_wK$ and $A\cap F_wV_0=\emptyset$ for some $w\in P$, or $A\cap(V_P\setminus V_0)\neq\emptyset$. 
\end{proof}

First, let us look at the upper bound asymptotic estimate of $\Gamma(x)$ (either $\Gamma_D(x)$ or $\Gamma_N(x)$) (Proposition \ref{prothm1}) under condition (A).
\begin{proof}[Proof of Proposition \ref{prothm1}]
For $x\geq 1$, write
\begin{equation*}
\Theta(x)=\{w=w_1w_2\cdots w_l\in W_*: \gamma^2_{w_1w_2\cdots w_{l-1}}x\geq 1>\gamma^2_wx\},
\end{equation*}
and let $\theta(x)=\#\Theta(x)$. Clearly, $\Theta(x)$ is a partition of $\Sigma$, and for $w=w_1w_2\cdots w_l\in\Theta(x)$,
\[\gamma_w<x^{-1/2}\leq \gamma_{w_1w_2\cdots w_{l-1}}\leq C\gamma_w,\]
with $C=(\min_{i\in S}\gamma_i)^{-1}$. Then, by iterating \eqref{defds}, we have
\[1=\sum_{w\in\Theta(x)}\gamma_w^{d_S}<\sum_{w\in\Theta(x)}x^{-d_S/2}\leq \sum_{w\in\Theta(x)}C^{d_S}\gamma_w^{d_S}=C^{d_S},\]
which yields $x^{d_S/2}<\theta(x)\leq C^{d_S}x^{d_S/2}$.

Let $\lambda\geq \lambda_0$ and $u\in E(\lambda)$. Then, for any $w\in\Theta(\lambda/\lambda_0)$ we have $\gamma^2_w \lambda<\lambda_0$, and by Proposition \ref{pro01}, $u\circ F_w\in E(\gamma^2_w \lambda)$, which gives $\#\mathrm{Extr}(u\circ F_w)\leq 1$ by condition (A).

By Lemma \ref{pro1}-(b), we have
\begin{equation*}
\begin{aligned}
\#\mathrm{Extr}(u)&\leq\sum_{w\in\Theta(\lambda/\lambda_0)}\#\mathrm{Extr}(u\circ F_w)+\#(V_{\Theta(\lambda/\lambda_0)}\setminus V_0)\\
&\leq\theta(\lambda/\lambda_0)+\#V_0\cdot\theta(\lambda/\lambda_0)=(L+1)\theta(\lambda/\lambda_0)\\
&\leq C^{d_S}(L+1)\lambda_0^{-d_S/2}\lambda^{d_S/2},
\end{aligned}
\end{equation*}
where $L=\#V_0$. Hence for $x\geq\lambda_0$, we have
\begin{equation*}
\begin{aligned}
\Gamma(x)&=\sup\{\#\mathrm{Extr}(u):u\in E_D(\lambda)\text{ or }E_N(\lambda),\ 0\leq\lambda\leq x\}\\
&\leq\sup\{\#\mathrm{Extr}(u):u\in E(\lambda),\ 0\leq\lambda\leq x\}\leq C^{d_S}(L+1)\lambda_0^{-d_S/2}x^{d_S/2}
\end{aligned}
\end{equation*}
and
\[\limsup_{x\to\infty}\frac{\Gamma(x)}{x^{d_S/2}}\leq C^{d_S}(L+1)\lambda_0^{-d_S/2}<\infty.\]
\end{proof}

\noindent{\bf Remark.} Condition (A) is not universally valid for all p.c.f. self-similar sets.  
For example, the {\it modified Koch curve} (Figure~\ref{figure2}) analyzed in \cite[Section 4.2]{Sh2} and \cite{M} satisfies $\#\mathrm{Extr}(u_\lambda)=\infty$ for every axially symmetric $u_\lambda\in E(\lambda)\setminus\{0\}$, independent of how small $\lambda>0$ is chosen. Indeed, for each singleton on the axis of symmetry, the local geometry together with $\lambda\neq 0$ forces the singleton to be an extreme set of $u_\lambda$, and there are infinitely many such singletons.
For the same reason, condition (A) also fails for Vicsek set-like fractals (see \cite[Page 95]{St6} for the definition), since each endpoint of the branches within the support of the eigenfunction is an extreme singleton, yielding $\#\mathrm{Extr}(u_\lambda)=\infty$ for any $\lambda>0$. By contrast, we shall verify in Section \ref{sec4} that condition (A) does hold for the Sierpinski gasket $\mathcal{SG}$ equipped with its canonical Laplacian.

\begin{figure}[h]
\begin{center}
\includegraphics[width=12.5cm]{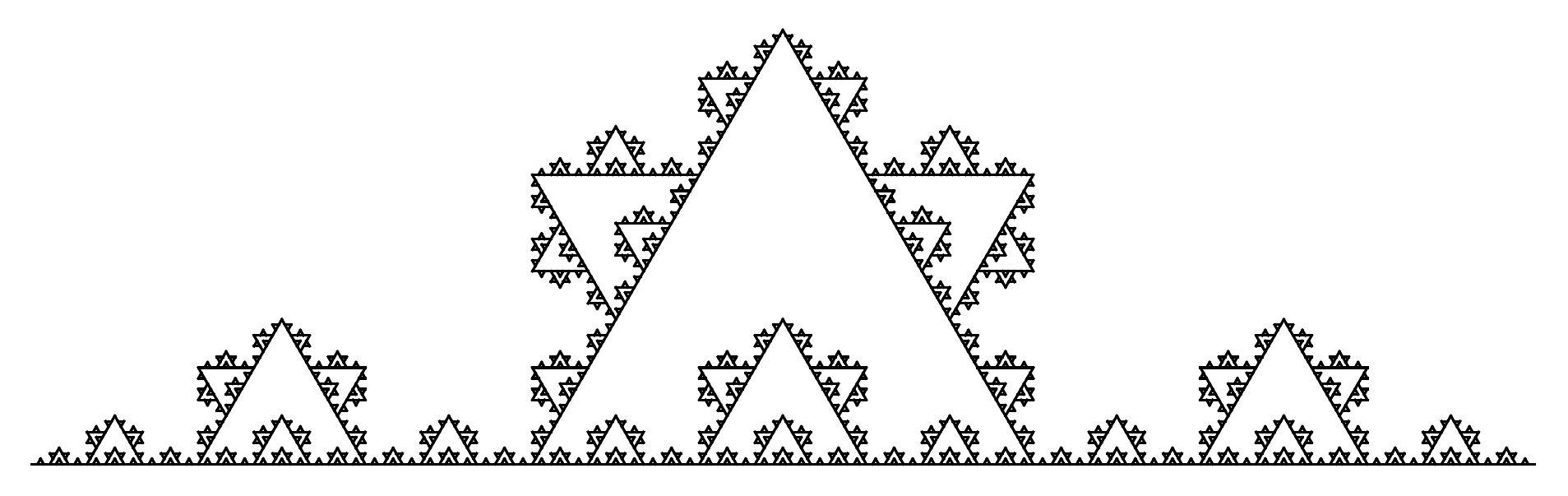}\\
\caption{The modified Koch curve.}
\label{figure2}
\end{center}
\end{figure}

In \cite[Section 4.3]{K4}, Kigami introduced the concept of pre-localized eigenfunctions.
Recall that a non-trivial $u$ is a pre-localized eigenfunction of $-\Delta_\mu$ if $u\in E_D(\lambda)\cap E_N(\lambda)$ for some $\lambda$.

A pre-localized eigenfunction produces genuinely ``localized" eigenfunctions whose support is confined to a single small cell of $K$.  In the lattice case, the existence of pre-localized eigenfunction is known to be equivalent to a jump in the integrated density of states  \cite[Theorem 4.3.4]{K4}, \cite[Theorem 4.4]{BK}. In the general case, a more restrictive symmetry assumption is required to guarantee existence; see \cite[Section 4.4]{K4}, \cite[Theorem 5.4]{BK}. It is known that this  requirement is satisfied for all affine nested fractals with $\# V_0\geq 3$ (including $\mathcal SG$ as a typical example) endowed with a symmetric invariant harmonic structure and a self-similar measure (\cite[Corollary 4.4.11]{K4}, \cite[Theorem 6.2]{BK}). 

Provided the existence of a pre-localized eigenfunction, we have a  lower bound estimate of $\Gamma(x)$, see Proposition \ref{prothm2}.

\begin{proof}[Proof of Proposition \ref{prothm2}] Suppose $u$ is a pre-localized eigenfunction belonging to some $\lambda>0$. Since $u\neq 0$ and $u|_{V_0}=\mathbf{0}$, $u$ has at least one local extremum in $K\setminus V_0$, which gives $\#\mathrm{Extr}(u)\geq 1$. 
For $w\in W_*$, define $P_wu$ by
\begin{equation*}
(P_wu)(x)=
\begin{cases}
u(F_w^{-1}(x))\quad\text{if }x\in F_wK,\\
0\quad\quad\quad\quad\quad\text{otherwise}.
\end{cases}
\end{equation*}
Then, by Proposition \ref{pro01},  we have $P_wu\in E_D(\gamma_w^{-2}\lambda)\cap E_N(\gamma_w^{-2}\lambda)$ and $\#\mathrm{Extr}(P_wu\circ F_w)=\#\mathrm{Extr}(u)\geq 1$.

Let $n\geq 1$ be an integer, denote $w^{(n)}=1^n2^n\cdots s^n\in W_{sn}$, and
\begin{equation*}
W(n)=\{w\in W_{sn}: w\text{ is a permutation of }w^{(n)}\}.
\end{equation*}
Clearly, for $w\in W(n)$, $\gamma_w=(\gamma_1\gamma_2\cdots\gamma_s)^n$. By Stirling's formula $\sqrt{2\pi}e^{-n}n^{n+1/2}<n!<\sqrt{2\pi}e^{-n}n^{n+1/2}e^{1/(12n)}$, we obtain
\begin{align*}
\#W(n)&=\frac{(sn)!}{(n!)^s}>\frac{\sqrt{2\pi}e^{-sn}(sn)^{sn+1/2}}{(\sqrt{2\pi})^se^{-sn}n^{sn+s/2}e^{s/(12n)}}\\
&\geq e^{-s/12}(2\pi)^{(1-s)/2}s^{sn+1/2}n^{(1-s)/2}.
\end{align*}

Consider $u_{n}:=\sum_{w\in W(n)}P_wu$. It is direct to check that $u_{n}\in E_D(\gamma^{-2}_{w^{(n)}}\lambda)\cap E_N(\gamma^{-2}_{w^{(n)}}\lambda)$. By Lemma \ref{pro1}-(b), 
\begin{equation}\label{eq2.3}
\Gamma(\gamma^{-2}_{w^{(n)}}\lambda)\geq\#\mathrm{Extr}(u_{n})\geq\sum_{w\in W(n)}\#\mathrm{Extr}\big((P_wu)\circ F_w\big)\geq\#W(n).
\end{equation}

For large $x>0$, let $n$ be the unique integer such that $\gamma^{-2}_{w^{(n)}}\lambda\leq x<\gamma^{-2}_{w^{(n+1)}}\lambda$. We have
\begin{align*}
\frac{\log \Gamma(x)}{\log x}&>\frac{\log \Gamma(\gamma^{-2}_{w^{(n)}}\lambda)}{\log \gamma^{-2}_{w^{(n+1)}}+\log\lambda}\geq\frac{\log\#W(n)}{-2(n+1)\log(\gamma_1\gamma_2\cdots\gamma_s)+\log\lambda}\\
&>\frac{\log(e^{-s/12}(2\pi)^{(1-s)/2}s^{1/2})+sn\log s+\frac{1-s}{2}\log n}{-2(n+1)\log(\gamma_1\gamma_2\cdots\gamma_s)+\log\lambda}\\
&\to\frac{s\log s}{-2\log(\gamma_1\gamma_2\cdots\gamma_s)}:=\kappa>0
\end{align*}
as $x\to\infty$. Note that by (\ref{defds}),
\[\log(\gamma_1\gamma_2\cdots\gamma_s)=\frac{s}{d_S}\log\sqrt[s]{\gamma_1^{d_S}\gamma_2^{d_S}\cdots\gamma_s^{d_S}}\leq-\frac{1}{d_S}s\log s\] 
gives $\kappa\leq d_S/2$.

Now we turn to the lattice case.  Let $T>0$ be the generator of $\sum_{i\in S}\mathbb{Z}\log\gamma_i$, so that $\log\gamma_i=-m_iT$ for $i\in S$, where $m_1, m_2, \cdots, m_s$ are positive integers with greatest common divisor $1$.

For an integer $n\geq1$, define
\[M(n)=\{w=w_1w_2\cdots w_l\in W_*: \gamma_w^{-2}\lambda= e^{2nT}\lambda\},\]
so that
\[M(n)=\{w=w_1w_2\cdots w_l\in W_*: \sum_{i=1}^lm_{w_i}=n\}.\]

Now consider $u_{n}:=\sum_{w\in M(n)}P_wu \in E_D(\lambda e^{2nT})\cap E_N(\lambda e^{2nT})$. Then, similarly to \eqref{eq2.3}, by Lemma \ref{pro1}-(b),
\[\Gamma(e^{2nT}\lambda)\geq\#\mathrm{Extr}(u_{n})\geq\sum_{w\in M(n)}\#\mathrm{Extr}\big((P_wu)\circ F_w\big)\geq\# M(n).\]
For large $x>0$, by choosing $n$ to be the unique integer satisfying $e^{2nT}\lambda\leq x<e^{2(n+1)T}\lambda$, we have
\[\frac{\Gamma(x)}{x^{d_S/2}}>\frac{\Gamma(e^{2nT}\lambda)}{e^{(n+1)Td_S}\lambda^{d_S/2}}\geq\frac{1}{e^{Td_S}\lambda^{d_S/2}}\cdot\frac{\# M(n)}{e^{nTd_S}}.\]
Finally, letting $x\to\infty$, it holds that 
\[\liminf_{x\to\infty}\frac{\Gamma(x)}{x^{d_S/2}}\geq\frac{1}{e^{Td_S}\lambda^{d_S/2}}\lim_{n\to\infty}\frac{\# M(n)}{e^{nTd_S}}=\frac{1}{e^{Td_S}\lambda^{d_S/2}}\big(\sum_{i\in S}m_i\gamma_i^{d_S}\big)^{-1}>0,\]
where the equality follows from \cite[Lemma 4.3.7]{K4}.
\end{proof}

\subsection{An equivalent condition for (A)}\label{subsec23}
Before proceeding, we introduce two projections from $\mathscr D_{\mu}$ to $\mathbb R^L$:
\[\tau^D:\mathscr{D}_\mu\to\mathbb{R}^L, \ \ \tau^D(u)=u|_{V_0},\]
\[\tau^N:\mathscr{D}_\mu\to\mathbb{R}^L, \ \ \tau^N(u)=\mathrm{d}u|_{V_0};\]
and for $\lambda\geq 0$, write
\begin{equation}\label{equationtau}
\tau_\lambda^D:=\tau^D|_{E(\lambda)},\ \ \tau_\lambda^N:=\tau^N|_{E(\lambda)}.
\end{equation}

The following is a basic observation.
\begin{lemma}\label{thm2.5}
$E(\lambda)$ is a linear space. $\tau_\lambda^D: E(\lambda)\to \mathbb R^L$ is a bijection for $\lambda\notin\Lambda^D$, and $\tau_\lambda^N: E(\lambda)\to \mathbb R^L$ is a bijection for $\lambda\notin\Lambda^N$. 
\end{lemma}
\begin{proof}
The linearity of $E(\lambda)$ is evident. It suffices to prove the statement for $\tau_\lambda^D$, since $\tau_\lambda^N$ is similar. For $\lambda\notin\Lambda^D$, we aim to prove that $\dim E(\lambda)\leq L$ and that $\tau_\lambda^D$ is a surjection.

Assume $\dim E(\lambda)>L$, noticing that $\tau_\lambda^D$ is linear, we have
\[\dim E_D(\lambda)=\dim\ker\tau_\lambda^D\geq\dim E(\lambda)-L>0,\]
hence $\lambda\in\Lambda^D$, contradicting  $\lambda\notin\Lambda^D$. Consequently, $\dim E(\lambda)\leq L$.

Fix $\lambda\notin\Lambda^D$, for each $\mathbf{a}\in\mathbb{R}^L$, we claim that the following problem has a solution $u\in\mathscr{D}_\mu$,
\begin{equation}\label{eqd1}
\begin{cases}
-\Delta_\mu u=\lambda u,\\
u|_{V_0}=\mathbf{a}.\\
\end{cases}
\end{equation}

Indeed, it is shown in \cite[Theorem 3.4.6, Corollary 3.4.7 and Theorem 3.7.9]{K4} that the Friedrichs extension of $-\Delta_\mu$ on $\mathscr{D}_{D,\mu}$, denoted by $H_D$, is a non-negative definite self-adjoint operator on $L^2(K,\mu)$, and its associated Dirichlet form is $(\mathscr{E},\mathscr{F}_0)$ with $\mathscr{F}_0=\{u\in\mathscr{F}:u|_{V_0}=\mathbf{0}\}$. Moreover, it has  compact resolvent with pure point spectrum  $\Lambda^D$. Choose some $v\in\mathscr{D}_\mu$ with $v|_{V_0}=\mathbf{a}$ and write $\tilde{u}=u-v$. Then, (\ref{eqd1}) can be rewritten as
\begin{equation}\label{eqd2}
\begin{cases}
H_D \tilde{u}-\lambda \tilde{u}=f,\\
\tilde{u}|_{V_0}=\mathbf{0},\\
\end{cases}
\end{equation}
where $f=\Delta_\mu v+\lambda v\in C(K)$. Since $\lambda\notin\Lambda^D$, the operator $H_D-\lambda$ is invertible and we have a solution $\tilde{u}\in L^2(K,\mu)$ to \eqref{eqd2}.

It is known in \cite[Page 133]{K4} that there exist $\varphi_i^D\in E_D(\lambda_i^D)$ such that $\{\varphi_i^D\}_{i\geq 1}$ is a complete orthonormal system for $L^2(K,\mu)$, so the solution $\tilde{u}$ has an expansion
\[\tilde{u}=\sum_{i=1}^\infty\langle \tilde{u},\varphi_i^D\rangle_\mu\varphi_i^D,\]
where $\langle \cdot,\cdot\rangle_\mu$ is the standard inner product in $L^2(K,\mu)$. Then (\ref{eqd2}) is equivalent to
\[(\lambda_i^D-\lambda)\langle \tilde{u},\varphi_i^D\rangle_\mu=\langle f,\varphi_i^D\rangle_\mu\quad \text{ for any } i\geq 1.\]
Therefore,
\[\tilde{u}=\sum_{i=1}^\infty\frac{\langle f,\varphi_i^D\rangle_\mu}{\lambda_i^D-\lambda}\varphi_i^D.\]

Since $\lambda\neq\lambda_i^D$ for all $i\geq 1$, we have $c_0:=\displaystyle \sup_{i\geq 1}\{\frac{\lambda_i^D}{(\lambda_i^D-\lambda)^2}\}<\infty$. Consequently, 
\[\mathscr{E}(\tilde{u},\tilde{u})=\langle H_D\tilde{u},\tilde{u}\rangle_\mu=\sum_{i=1}^\infty\lambda_i^D\langle\tilde{u},\varphi_i^D\rangle_\mu^2=\sum_{i=1}^\infty\frac{\lambda_i^D}{(\lambda_i^D-\lambda)^2}\langle f,\varphi_i^D\rangle_\mu^2\leq c_0\Vert f\Vert^2_{{L^2}(K,\mu)}<\infty,\]
so that $\tilde{u}\in\mathscr{F}\subset C(K)$. Thus,
\[H_D\tilde{u}=\lambda\tilde{u}+f\in C(K),\]
implying $\tilde{u}\in \mathscr{D}_{D,\mu}$, and therefore $u=\tilde{u}+v\in E(\lambda)$ is a solution to 
(\ref{eqd1}). So the claim holds.

From the claim, $\tau_\lambda^D$ is a surjection. $\tau_\lambda^D$ is also injective since it is linear and $\dim E(\lambda)\leq L$. This completes the proof.
\end{proof}

\noindent{\bf Remark.} As an immediate consequence of Lemma \ref{thm2.5}, 
$\mathrm{dim}E(\lambda)=L$ for $\lambda\notin\Lambda^D\cap\Lambda^N$. For $\lambda\notin \Lambda^D$, the map $\tau_\lambda^D$ establishes a one-to-one correspondence between functions in $E(\lambda)$ and their boundary values. Similarly, for $\lambda\notin \Lambda^N$, the map $\tau_\lambda^N$ establishes a one-to-one correspondence between functions in $E(\lambda)$ and their boundary Neumann derivatives. 

For $0<\lambda<\lambda_1^D$, by Lemma \ref{thm2.5}, $\tau_{\lambda}^D: E(\lambda)\to \mathbb R^L$ is invertible. Hence for such $\lambda$, by introducing 
\begin{equation}\label{defT}
T_\lambda^i(\mathbf a)=\tau_{\gamma^2_i\lambda}^D\Big((\tau_\lambda^D)^{-1}(\mathbf a)\circ F_i\Big)\quad \text{ for any }\mathbf a\in \mathbb R^L,
\end{equation}
it is direct to see that $T_\lambda^i:\mathbb R^L\to \mathbb R^L$ is a linear map.
Note that $T_\lambda^i$ may not be invertible.

For $0<\lambda<\lambda_1^D$, define
\begin{equation}\label{defab}
\begin{aligned}
&\mathbf{A}_\lambda=\{\mathbf{a}\in\mathbb{R}^L:(\tau_\lambda^D)^{-1}(\mathbf{a}) \text{ has at least one extreme set}\},\\
&\mathbf{B}_\lambda=\{\mathbf{a}\in\mathbb{R}^L:(\tau_\lambda^D)^{-1}(\mathbf{a})\text{ has exactly one extreme set }A \text{ with }A\cap V_1\neq\emptyset\}.\\
\end{aligned}
\end{equation}

Then we have the following criterion to verify condition (A). (Recall condition (A) in Section \ref{subsec11}).

\begin{proposition}\label{pro2}
The condition (A) is satisfied if and only if for every sufficiently small $\lambda>0$,
\begin{equation} \label{eq1}
\begin{aligned}
&\mathbf{A}_\lambda=\big(\bigcup_{i\in S}(T_\lambda^i)^{-1}(\mathbf{A}_{\gamma^2_i\lambda})\big)\cup\mathbf{B}_\lambda,\quad\text{with}\\
(T_\lambda^i&)^{-1}(\mathbf{A}_{\gamma^2_i\lambda}),\ i\in S\text{ and }\mathbf{B}_\lambda\text{ pairwise disjoint.}
\end{aligned}
\end{equation}
\end{proposition}
\begin{proof}
Write
\begin{equation*}
\begin{aligned}
&\mathbf{A}_{\lambda,i}=\{\mathbf{a}\in\mathbb{R}^L:(\tau_\lambda^D)^{-1}(\mathbf{a})\text{ has at least one extreme set }A\subset F_iK\text{ with }A\cap V_1=\emptyset\},\\
&\mathbf{B}'_\lambda=\{\mathbf{a}\in\mathbb{R}^L:(\tau_\lambda^D)^{-1}(\mathbf{a})\text{ has at least two extreme sets }A\text{ with }A\cap V_1\neq\emptyset\}.\\
\end{aligned}
\end{equation*}
Then we have the natural decomposition
\[\mathbf{A}_\lambda=\big(\bigcup_{i\in S}(\mathbf{A}_{\lambda,i})\big)\cup\mathbf{B}_\lambda\cup\mathbf{B}'_\lambda.\]

If $\mathbf{a}\in\mathbf{A}_{\lambda,i}$ for some $i\in S$ and $(\tau_\lambda^D)^{-1}(\mathbf{a})$ has at least one extreme set $A$ with $A\subset F_iK$ and $A\cap V_1=\emptyset$, then $(\tau_\lambda^D)^{-1}(\mathbf{a})\circ F_i$ has at least one extreme set, so $\mathbf{a}\in(T_\lambda^i)^{-1}(\mathbf{A}_{\gamma^2_i\lambda})$. Conversely,  if $\mathbf{a}\in(T_\lambda^i)^{-1}(\mathbf{A}_{\gamma^2_i\lambda})$, then $T_\lambda^i(\mathbf{a})=\tau_{\gamma^2_i\lambda}^D\big((\tau_\lambda^D)^{-1}(\mathbf a)\circ F_i\big)\in\mathbf{A}_{\gamma^2_i\lambda}$, so $(\tau_\lambda^D)^{-1}(\mathbf{a})\circ F_i$ has at least one extreme set. It follows that $(\tau_\lambda^D)^{-1}(\mathbf{a})$ has at least one extreme set $A$ with $A\subset F_iK$ and $A\cap V_1=\emptyset$, that is, $\mathbf{a}\in \mathbf{A}_{\lambda,i}$. Therefore,  $\mathbf{A}_{\lambda,i}=(T_\lambda^i)^{-1}(\mathbf{A}_{\gamma^2_i\lambda})$.

Suppose that condition (\ref{eq1}) is satisfied (note that this implies $\mathbf{B}'_\lambda=\emptyset$), then for any $u=u_\lambda\in E(\lambda)$ with small $\lambda$, the extreme sets of $u$ must be in one of the following mutually exclusive cases:

(1). $\tau_\lambda^D(u)\notin\mathbf{A}_\lambda$, so $\#\mathrm{Extr}(u)=0$;

(2). $\tau_\lambda^D(u)\in\mathbf{B}_\lambda$, so $\#\mathrm{Extr}(u)=1$;

(3). $\tau_\lambda^D(u)\in(T_\lambda^i)^{-1}(\mathbf{A}_{\gamma^2_i\lambda})$
for exactly one $i\in S$, then $\#\mathrm{Extr}(u)=\#\mathrm{Extr}(u\circ F_i)\geq 1$.

In case (3), observing that $u\circ F_i\in E(\gamma^2_i\lambda)$, we proceed by analyzing the extreme sets of $u\circ F_i$  instead, and repeat the above procedure iteratively. There are two possibilities: one possibility is that there is some $w=w_1w_2\ldots w_m\in W_*$ such that  $u\circ F_w$ is in case (2), but $u\circ F_{w_1w_2\ldots w_k}$ is in case (3) for each $k<m$, and we obtain $\#\mathrm{Extr}(u)=\#\mathrm{Extr}(u\circ F_{w_1})=\cdots=\#\mathrm{Extr}(u\circ F_{w})=1$; the other possibility is that there exists a unique $\omega\in\Sigma$ such that $\{\Pi(\omega)\}$ is the only extreme set of $u$, and we still have $\#\mathrm{Extr}(u)=1$. This implies that condition (A) holds.

Conversely, suppose (\ref{eq1}) fails, then there exists $\lambda>0$ arbitrarily small and $\mathbf{a}\in\mathbb{R}^L$ such that $\mathbf{a}\in\mathbf{B}'_\lambda$ or $\mathbf{a}$ belongs to at least two of the sets $(T_\lambda^i)^{-1}(\mathbf{A}_{\gamma^2_i\lambda}),\ i\in S$ and $\mathbf{B}_\lambda$. This gives  $\#\mathrm{Extr}((\tau_\lambda^D)^{-1}(\mathbf{a}))\geq 2$ in both situations, hence condition (A) fails.
\end{proof}

\noindent{\bf Remark.} By Lemma \ref{pro1}-(a), $\mathbf{A}_\lambda$ and $\mathbf{B}_\lambda$ are cones, i.e. if $\mathbf{a}\in \mathbf{A}_\lambda\ (\text{or }\mathbf{B}_\lambda)$, then $k\mathbf{a}\in \mathbf{A}_\lambda \ (\text{or }\mathbf{B}_\lambda)$ for any $k\neq 0$. Let $\pi:\mathbb{R}^L\setminus\{\mathbf{0}\}\to\mathbb{R}\mathrm{P}^{L-1}$ be the canonical projection that maps each point in $\mathbb{R}^L\setminus\{\mathbf{0}\}$ to the line through the origin it generates. 
For $\mathbf{a},\mathbf{b}\in\mathbb{R}^L$, define $\mathbf{a}\sim\mathbf{b}$ if there exists $k\in\mathbb{R}\setminus\{0\}$ such that $\mathbf{a}=k\mathbf{b}$, then $\pi$ can be regarded as a bijection from $(\mathbb{R}^L\setminus\{\mathbf{0}\})/\sim$ to $\mathbb{R}\mathrm{P}^{L-1}$.
Write \[\mathcal{A}_\lambda=\pi(\mathbf{A}_\lambda),\quad \mathcal{B}_\lambda=\pi(\mathbf{B}_\lambda),\] and define
\begin{equation}\label{defabT}
\mathcal{T}_\lambda^i:\mathbb{R}\mathrm{P}^{L-1}\to\mathbb{R}\mathrm{P}^{L-1} \quad \text{by}\quad  \mathcal{T}_\lambda^i=\pi\circ T_\lambda^i\circ\pi^{-1}.
\end{equation}
Then, the condition (\ref{eq1}) is equivalent to
\begin{equation}\tag{\ref{eq1}*}\label{eq2}
\begin{aligned}
&\mathcal{A}_\lambda=\big(\bigcup_{i\in S}(\mathcal{T}_\lambda^i)^{-1}(\mathcal{A}_{\gamma^2_i\lambda})\big)\cup\mathcal{B}_\lambda,\quad\text{with}\\
(\mathcal{T}_\lambda^i&)^{-1}(\mathcal{A}_{\gamma^2_i\lambda}),\ i\in S\text{ and }\mathcal{B}_\lambda\text{ pairwise disjoint.}
\end{aligned}
\end{equation}
Consequently, condition (A) is equivalent to \eqref{eq2}. 

In Section \ref{sec4}, we will verify condition (A) through \eqref{eq2} for the Sierpinski gasket $\mathcal {SG}$ equipped with its canonical Laplacian.

\section{Eigenfunctions on the Sierpinski Gasket}\label{sec3}

From now on, we specialize to the Sierpinski gasket $\mathcal {SG}$ equipped with its canonical Laplacian, see \cite{K1, K4}. In this section we first briefly review the recipe of spectral decimation on $\mathcal{SG}$, then state two preparatory results, Theorems \ref{thm4} and \ref{thm5}, whose proofs are postponed to Sections \ref{sec4} and \ref{sec5}.

Let $S=\{1,2,3\}$ and $V_0=\{p_1,p_2,p_3\}$ be a set of vertices of an equilateral triangle in $\mathbb{R}^2$. Set $F_i(x)=(x-p_i)/2+p_i$ for $i\in S$. The {\it Sierpinski gasket} $\mathcal{SG}$ is the attractor of the i.f.s. $\{F_i\}_{i\in S}$. 
Denote $V_1\setminus V_0=\{p_{23},p_{31},p_{12}\}$ with $p_{ij}=F_ip_j=F_jp_i$ for $ij\in S_1:=\{23,31,12\}$, see Figure \ref{figure3}.

\begin{figure}[h]
\begin{center}
\includegraphics[width=5.5cm]{SG2.pdf}\hspace{2cm}
\includegraphics[width=5.35cm]{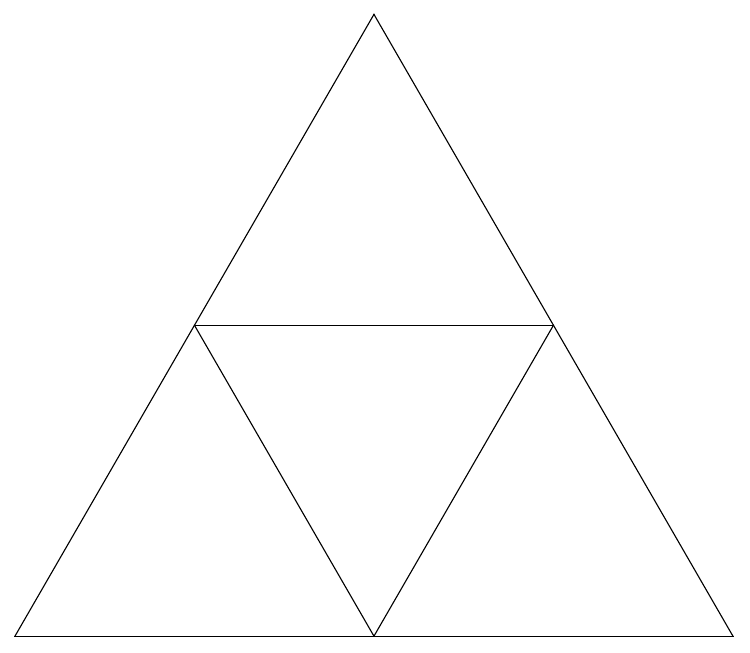}\\
\caption{The Sierpinski gasket $\mathcal{SG}$ and the set $V_1$.}
\label{figure3}
\end{center}
\begin{picture}(0,0) \thicklines
\put(105,172){$p_1$}
\put(23,35){$p_2$}
\put(185,35){$p_3$}
\put(55,107){$p_{12}$}
\put(149,107){$p_{31}$}
\put(103,31){$p_{23}$}
\end{picture}
\end{figure}

Let $D=\begin{pmatrix}
-2&1&1\\
1&-2&1\\
1&1&-2
\end{pmatrix}$ and $\mathbf{r}=(3/5,3/5,3/5)$, then $(D,\mathbf{r})$ is a regular harmonic structure on $\mathcal{SG}$. For $m\geq1$, write $p\sim_mq$ if $p\neq q\in V_m$ and there exists $w\in W_m$ such that $p,q\in F_wV_0$. Note that $\#\{q:q\sim_mp\}=4$ for $p\in V_m\setminus V_0$ and $\#\{q:q\sim_mp\}=2$ for $p\in V_0$. Define $\Delta_m:l(V_m)\to l(V_m)$ by
\begin{equation}\label{defdeltam}
(\Delta_m u)(p)=\sum_{q\sim_m p}(u(q)-u(p)),
\end{equation}
then $H_m=(5/3)^m\Delta_m$.

Denote $\mu$ as the standard self-similar measure on $\mathcal{SG}$ with $\mu_i=1/3$ for all $i\in S$. Recall (\ref{equation2.1}), it is easy to calculate that $\int_\mathcal{SG}\psi_{m,p}\mathrm{d}\mu=2/3^{m+1}$ for $p\in V_m\setminus V_0$ and $\int_\mathcal{SG}\psi_{m,p}\mathrm{d}\mu=1/3^{m+1}$ for $p\in V_0$, which gives 
\begin{equation}\label{defdelta}
\mu_{m,p}^{-1}(H_mu)(p)=\frac{3}{2}5^m(\Delta_mu)(p),\quad\text{for any }p\in V_m\setminus V_0.
\end{equation}

Denote by $\Delta$ (omit the subscript $\mu$ for simplicity) the associated Laplacian, and call it the canonical Laplacian on $\mathcal{SG}$. Denote by $\mathscr{D}$ the domain of $\Delta$. The Neumann derivative of $u\in\mathscr D$ at $p_i$ is
\begin{equation}\label{defdu}
\begin{aligned}
(\mathrm{d}u)_{p_i}&=\lim_{m\to\infty}-(H_mu)(p_i)\\
&=\lim_{m\to\infty}\Big(\frac{5}{3}\Big)^m\big(2u(p_i)-u(F_i^mp_j)-u(F_i^mp_k)\big).
\end{aligned}
\end{equation}

Additionally, we have $\gamma^2_i=r_i\mu_i=1/5$ for $i\in S$ and $d_S=\log 9/\log 5$, so that it is in the lattice case (recall the last paragraph in Section \ref{subsec21}).

\subsection{The spectral decimation method}

In this subsection, we recall the spectral decimation method on $\mathcal{SG}$ due to Shima and Fukushima. Details can be found in \cite{FS,Sh,St6}.

Let $\Phi(x)=x(5-x)$. Denote the two branches of $\Phi^{-1}$ by
\begin{equation*}
\varphi_{-1}(x)=\frac{1}{2}(5-\sqrt{25-4x}),\quad \varphi_1(x)=\frac{1}{2}(5+\sqrt{25-4x}) \quad\text{ on $(-\infty,25/4]$}.
\end{equation*}
Define
\begin{equation*}
\psi(x)=\frac{3}{2}\lim_{m\to\infty}5^m\varphi_{-1}^m(x),
\end{equation*}
where $\varphi_{-1}^m$ denotes the $m$-th iteration of $\varphi_{-1}$, so that $\psi$ is a strictly increasing analytic function on $(-\infty,25/4)$. In addition, $5\psi(\varphi_{-1}(x))=\psi(x)$.

For any $\eps=\eps_1\eps_2\cdots\eps_n\in\{-1,1\}^n$ of length $|\eps|:=n\geq 0$, set 
\[\varphi_\eps=\varphi_{\eps_n}\circ\varphi_{\eps_{n-1}}\circ\cdots\circ\varphi_{\eps_1}\]
with the convention that  $\varphi_\varnothing=\mathrm{id}$. Define
\begin{equation*}
\Psi(m,\eps,x)=5^{m+|\eps|}\psi\circ\varphi_\eps(x).
\end{equation*}

\begin{proposition}\label{thm31}(\cite[Lemma 2.1 and Proposition 2.2]{Sh})
Let $u\in E(\lambda)\setminus\{0\}$ for $\lambda\geq 0$.

(a). There exists a minimal integer $m_0\geq 1$, called the {\rm level of birth}, such that
\begin{equation*}
-\Delta_m u|_{V_m}=\lambda_m u|_{V_m} \text{\quad on }V_m\setminus V_0
\end{equation*}
holds for any $m\geq m_0$ with some $\lambda_m$ and

there exists $m_1>m_0$, called the {\rm level of fixation}, together with $\eps=\eps_1\eps_2\cdots\eps_n\in\{-1,1\}^n$ of length $n=|\eps|=m_1-m_0-1$ (with $\eps_n=1$ if $n\geq 1$) such that
\begin{equation*}
\lambda_{m+1}=
\begin{cases}
\varphi_{\eps_{m-m_0+1}}(\lambda_m),\quad &\text{if \ } m_0\leq m<m_1-1,\\
\varphi_{-1}(\lambda_m), \quad &\text{if \ } m\geq m_1-1.\\
\end{cases}
\end{equation*}
In particular, $\lambda_m\notin\{2,5,6\}$ for all $m>m_0$;

and the eigenvalue $\lambda$ is recovered via
\begin{equation*}
\lambda=\frac{3}{2}\lim_{m\to\infty}5^m\lambda_m=\Psi(m_0,\eps,\lambda_{m_0}).
\end{equation*}

(b). Conversely, for $m\geq 1$, if $-\Delta_{m+1} u|_{V_{m+1}}=\lambda_{m+1} u|_{V_{m+1}}$ on $V_{m+1}\setminus V_0$ with some $\lambda_{m+1}\notin\{2,5,6\}$, then $-\Delta_{m} u|_{V_{m}}=\lambda_{m} u|_{V_{m}}$ on $V_m\setminus V_0$ with $\lambda_{m}=\Phi(\lambda_{m+1})$.

(c). The $\lambda$-eigenfunction $u$ is uniquely determined by $u|_{V_{m_0}}$ through the extension rule:
\begin{equation}\label{eqdeci}
u(p_{ij}^w)=\frac{(4-\lambda_{m+1})(u(p_i^w)+u(p_j^w))+2u(p_k^w)}{(2-\lambda_{m+1})(5-\lambda_{m+1})}
\end{equation}
for each $w\in W_m$ with $m\geq m_0$ and distinct $i,j,k\in S$, where $p_i^w=F_wp_i$ and $p_{ij}^w=F_wp_{ij}$.
\end{proposition}

The following result gives the Dirichlet and Neumann spectra of $\mathcal{SG}$.

\begin{proposition}\label{thm32}(\cite[Theorem 2.1 and Theorem 3.1]{Sh})
If $u\in E_D(\lambda)\setminus\{0\}$ (or $E_N(\lambda)\setminus\{0\}$), then there exists $m_0$ as its level of birth and $\eps=\eps_1\eps_2\cdots\eps_n\in\{-1,1\}^n$ of length $n\geq 0$ (with $\eps_n=1$ if $n\geq 1$) such that

The Dirichlet case: for $u\in E_D(\lambda)$,

(D2). $m_0=1,\ \lambda_{m_0}=2,\ \lambda=\Psi(1,\eps,2)$, and $\dim E_D(\lambda)=1$;

(D5). $m_0\geq 1,\ \lambda_{m_0}=5,\ \lambda=\Psi(m_0,\eps,5)$, and $\dim E_D(\lambda)=(3^{m_0-1}+3)/2$;

(D6). $m_0\geq 2,\ \lambda_{m_0}=6,\ \lambda_{m_0+1}=3,\ \lambda=\Psi(m_0+1,\eps,3)$, and $\dim E_D(\lambda)=(3^{m_0}-3)/2$;

The Neumann case: for $u\in E_N(\lambda)$,

(N0). $m_0=1,\ \lambda_{m_0}=0,\ \lambda=0$, and $\dim E_N(\lambda)=1$;

(N5). $m_0\geq 2,\ \lambda_{m_0}=5,\ \lambda=\Psi(m_0,\eps,5)$, and $\dim E_N(\lambda)=(3^{m_0-1}-1)/2$;

(N6). $m_0\geq 1,\ \lambda_{m_0}=6,\ \lambda_{m_0+1}=3,\ \lambda=\Psi(m_0+1,\eps,3)$, and $\dim E_N(\lambda)=(3^{m_0}+3)/2$; 

(N6'). $m_0=1,\ \lambda_{m_0}=3,\ \lambda=\Psi(1,\eps,3)$, and $\dim E_N(\lambda)=2$.
\end{proposition}



\begin{remark}\label{re33}
From the above propositions, it is direct to verify that

(1). $\lambda_1^D=5\psi(2)=\psi(6)$, $\lambda_1^N=0$ and $\lambda_2^N=5\psi(3)$;

(2). for any $u\in E(\lambda)\setminus\{0\}$ with $0<\lambda<\lambda_1^D$, we have $m_0=1, \eps=\varnothing$ and $\lambda_m=\psi^{-1}(5^{-m}\lambda)\in(0,2)$ for each $m\geq 1$.
\end{remark}







\subsection{Statements of Theorems \ref{thm4} and \ref{thm5}}

First, let us focus on eigenfunctions belonging to small eigenvalues. Since $\lambda_1^D=5\psi(2)<5\psi(3)=\lambda_2^N$, we have $\lambda\notin\Lambda^D\cup\Lambda^N$ for $0<\lambda<\lambda_1^D$. Therefore, by Lemma \ref{thm2.5}, for such $\lambda$, both $\tau_\lambda^D$ and  $\tau_\lambda^N$ (recall \eqref{equationtau}) are invertible, hence $\tau_\lambda^D\circ(\tau_\lambda^N)^{-1}:\mathbb R^3\to\mathbb R^3$ is a linear bijection.

Denote
\begin{equation*}
\mathbf{C}=\{(a_1,a_2,a_3)\in\mathbb{R}^3:\text{either }a_i>0\text{ for all }i, \text{ or }a_i<0\text{ for all }i\}
\end{equation*}
and $\mathbf{C}_\lambda=(\tau_\lambda^D\circ(\tau_\lambda^N)^{-1})(\mathbf{C})$. Obviously both $\mathbf{C}$ and $\mathbf{C}_\lambda$ are cones.


For the projective plane $\mathbb{R}\mathrm{P}^2$, we take the decomposition $\mathbb{R}\mathrm{P}^2=\mathbb{R}^2\cup L_\infty$, where $L_\infty:=\mathbb{R}\mathrm{P}^2\setminus\mathbb{R}^2$ denote the line at infinity of $\mathbb{R}\mathrm{P}^2$.

Define $\pi:\mathbb{R}^3\setminus\{\mathbf{0}\}\to\mathbb{R}\mathrm{P}^2$ by
\begin{equation}\label{defpi}
\pi(\mathbf{x})=
\begin{cases}
(\mathbf{1}^\mathrm{t}\mathbf{x})^{-1}(Q^\mathrm{t}\mathbf{x}),\quad &\text{if \ } \mathbf{1}^\mathrm{t}\mathbf{x}\neq 0,\\
[Q^\mathrm{t}\mathbf{x}]_\infty, \quad &\text{if \ } \mathbf{1}^\mathrm{t}\mathbf{x}=0,\\
\end{cases}
\end{equation}
where $\mathbf{1}=(1,1,1)^\mathrm{t}$, $\displaystyle Q^\mathrm{t}=\frac{1}{2}\begin{pmatrix}
0&-\sqrt{3}&\sqrt{3}\\
2&-1&-1
\end{pmatrix}$, and $[\boldsymbol{\xi}]_\infty\in L_\infty$ denotes the point at infinity in the direction $\boldsymbol{\xi}$. Note that $\pi$ maps $\mathbf{1}$ to the origin $\boldsymbol{\theta}$ in $\mathbb R\mathrm{P}^2$, and sends 
\begin{equation*}
(1,0,0)^\mathrm t\mapsto(0,1)^\mathrm t, \quad (0,1,0)^\mathrm t \mapsto (-\frac{\sqrt{3}}2,-\frac12)^\mathrm t, \quad (0,0,1)^\mathrm t \mapsto (\frac{\sqrt{3}}2,-\frac12)^\mathrm t.
\end{equation*}

For $\mathbf{A}_\lambda,\mathbf{B}_\lambda$ defined in (\ref{defab}) for $\mathcal{SG}$, write $\mathcal{A}_\lambda=\pi(\mathbf{A}_\lambda),\ \mathcal{B}_\lambda=\pi(\mathbf{B}_\lambda)$. Write $\mathcal{C}_\lambda=\pi(\mathbf{C}_\lambda)$, and $\mathcal{T}_\lambda^i=\pi\circ T_\lambda^i\circ\pi^{-1}$.
For simplicity, write $\tau=\pi\circ\tau^D$, and for $\alpha\in(0,6)$, denote by
\begin{equation}\label{defmd}
\mathcal{D}_\alpha=\Big\{(\xi^{(1)},\xi^{(2)})^\mathrm{t}\in\mathbb{R}\mathrm{P}^2:
\frac{\alpha}{2(6-\alpha)}>\xi^{(2)}>\sqrt{3}|\xi^{(1)}|-\frac{\alpha}{6-\alpha}\Big\}
\end{equation}
the open equilateral triangle centered at $\boldsymbol{\theta}$ with vertices
\begin{equation*}
\boldsymbol{\zeta}_{\alpha,23}=\frac{\alpha}{6-\alpha}\begin{pmatrix}0\\-1\end{pmatrix},\quad \boldsymbol{\zeta}_{\alpha,31}=\frac{\alpha}{2(6-\alpha)}\begin{pmatrix}\sqrt{3}\\1\end{pmatrix},\quad
\boldsymbol{\zeta}_{\alpha,12}=\frac{\alpha}{2(6-\alpha)}\begin{pmatrix}-\sqrt{3}\\1\end{pmatrix}.
\end{equation*}
Moreover, let $\mathcal{L}_{\alpha,ij}$ denote the open line segment joining $\boldsymbol{\theta}$ and $\boldsymbol{\zeta}_{\alpha,ij}$, and $\mathcal{G}_{\alpha,i}$ denote the open triangle with vertices $\boldsymbol{\theta},\ \boldsymbol{\zeta}_{\alpha,ki},\ \boldsymbol{\zeta}_{\alpha,ij}$ for distinct $i,j,k\in S$.

The next theorem is one of the main results in this paper, which gives explicit expressions for $\mathcal{A}_\lambda$ and $\mathcal{B}_\lambda$.

\begin{theorem}\label{thm4}
For $0<\lambda<\lambda_1^D$, we have
\begin{equation}\label{eq3}
\begin{aligned}
\mathcal{A}_\lambda=\ & \mathcal{C}_\lambda=\mathcal{D}_{\psi^{-1}(\lambda)},\quad \mathcal{B}_\lambda=\big(\bigcup_{ij\in S_1}\mathcal{L}_{\psi^{-1}(\lambda),ij}\big)\cup\{\boldsymbol{\theta}\},\quad\text{and}\\
&(\mathcal{T}_\lambda^i)^{-1}(\mathcal{A}_{5^{-1}\lambda})=
\mathcal{G}_{\psi^{-1}(\lambda),i}\quad\text{for }i\in S.
\end{aligned}
\end{equation}
In particular, condition (A) holds.
\end{theorem}

\begin{figure}[h]
\begin{center}
\includegraphics[width=9cm]{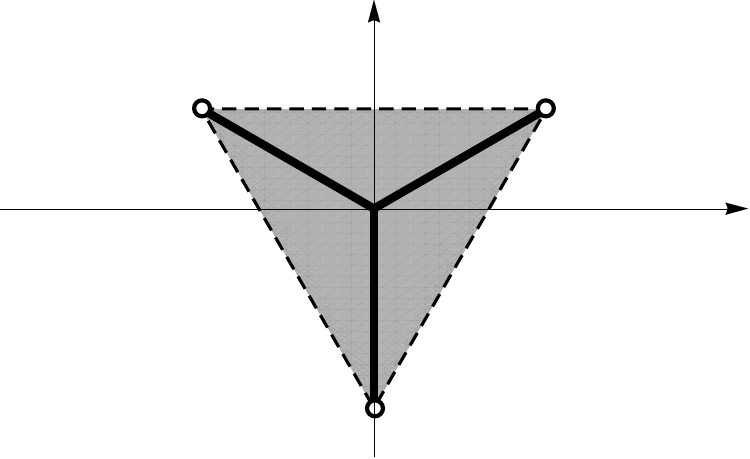}\\
\caption{$\mathcal{A}_\lambda$ (the shaded open equilateral triangle), $\mathcal{B}_\lambda$ (the thickened Y-shaped line segments, with endpoints removed), and $\mathcal{G}_{i}:=\mathcal{G}_{\psi^{-1}(\lambda),i}=(\mathcal{T}_\lambda^i)^{-1}\mathcal{A}_{5^{-1}\lambda}$, $i\in S$ (the three small open triangles).}
\label{figure4}
\end{center}
\begin{picture}(0,0) \thicklines
\put(-5,167){$\mathcal{G}_{1}$}
\put(-25,145){$\mathcal{G}_{2}$}
\put(7,127){$\mathcal{G}_{3}$}
\put(123,160){$\xi^{(1)}$}
\put(7,217){$\xi^{(2)}$}
\put(3,140){$\boldsymbol{\theta}$}
\put(5,75){$\boldsymbol{\zeta}_{\psi^{-1}(\lambda),23}$}
\put(65,185){$\boldsymbol{\zeta}_{\psi^{-1}(\lambda),13}$}
\put(-73,193){$\boldsymbol{\zeta}_{\psi^{-1}(\lambda),12}$}
\end{picture}
\end{figure}
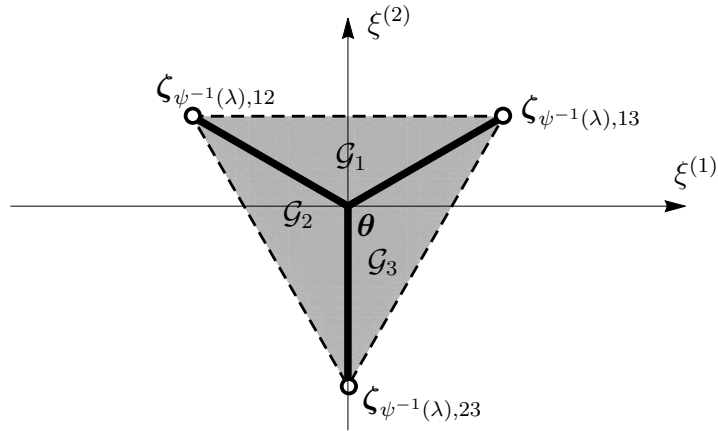

\noindent{\bf Remark 1.} As shown in Figure \ref{figure4}, $\mathcal{A}_\lambda=\big(\bigcup_{i\in S}(\mathcal{T}_\lambda^i)^{-1}(\mathcal{A}_{5^{-1}\lambda})\big)\cup\mathcal{B}_\lambda$ is a disjoint union, which gives \eqref{eq2}. Consequently, by Proposition \ref{pro2}, condition (A) holds. 

\noindent{\bf Remark 2.} Actually, the proof of 
Theorem \ref{thm4} provides an algorithm to locate the extreme set of any $u\in E(\lambda)\setminus\{0\}$ with $0<\lambda<\lambda_1^D$:

(1). $\#\mathrm{Extr}(u)=0$ when $\tau(u)\notin\mathcal{A}_\lambda$;

(2). $\#\mathrm{Extr}(u)=1$ when $\tau(u)\in\mathcal{A}_\lambda$;

\noindent furthermore, for the unique extreme set $A$ in case (2), we have

(2-1). $A=\overline{p_{23}p_{31}p_{12}}$ when $\tau(u)=\boldsymbol{\theta}$,

(2-2). $A=\{p_{ij}\}$ when $\tau(u)\in\mathcal{L}_{\psi^{-1}(\lambda),ij}$ for some $ij\in S_1$,

(2-3). $A\subset F_i\mathcal{SG}$ with $A\cap V_1=\emptyset$ when $\tau(u)\in \mathcal{G}_{\psi^{-1}(\lambda),i}$ for some $i\in S$,

\noindent where $\overline{p_{23}p_{31}p_{12}}$ denotes the boundary of the triangle with vertices $p_{23},p_{31},p_{12}$, see Figure \ref{figure3} and \ref{figure4};

\noindent when (2-3) happens, one can repeat the procedure for $u\circ F_i$ iteratively until the location of the extreme set is arrived.

We will prove Theorem \ref{thm4} in Section \ref{sec4}.


Next, let us look at a certain class of eigenfunctions on $\mathcal{SG}$.



Let $\lambda_0\in(0,6)$ and $\mathbf{a}\in\mathbb{R}^3\setminus\{\mathbf 0\}$. For an integer $n\geq 0$, and $\eps=\eps_1\eps_2\cdots\eps_n\in\{-1,1\}^n$ of length $n$ with $\eps_n=1$ if $n\geq 1$, define a sequence $\{\lambda^\eps_m\}_{m\geq 0}$ by
\begin{equation*}
\lambda^{\eps}_0=\lambda_0,\quad\lambda^{\eps}_{m+1}=
\begin{cases}
\varphi_{\eps_{m+1}}(\lambda^{\eps}_m),\quad &\text{if \ } 0\leq m<n,\\
\varphi_{-1}(\lambda^{\eps}_m), \quad &\text{if \ } m\geq n,\\
\end{cases}
\end{equation*}
and define a continuous function $u^\eps$ on $\mathcal{SG}$ by $u^\eps|_{V_0}=\mathbf{a}$ and
\begin{equation*}
u^{\eps}(p_{ij}^w)=\frac{(4-\lambda^{\eps}_{m+1})(u^{\eps}(p_i^w)+u^{\eps}(p_j^w))+2u^{\eps}(p_k^w)}{(2-\lambda^{\eps}_{m+1})(5-\lambda^{\eps}_{m+1})}
\end{equation*}
for each $w\in W_m$ and distinct $i,j,k\in S$, where $p_i^w=F_w(p_i)$, and $p_{ij}^w=F_w(p_{ij})$.
By the spectral decimation algorithm (Proposition \ref{thm31}), it is known that $u^\eps\in E(\lambda^{\eps})$ with $\lambda^{\eps}=\Psi(0,\eps,\lambda_0)$. Note that when $\eps=\varnothing$, it is obvious that $\lambda^\varnothing <\lambda_1^D$, so  $\#\mathrm{Extr}(u^\varnothing)=0$ or $1$ by Theorem \ref{thm4}.  

\begin{theorem}\label{thm5}
For  $\eps\in\{-1,1\}^n$ of length $n\geq 1$ with $\eps_n=1$, we have 
\begin{equation}\label{equn}
c^{-1}(\lambda^{\eps})^{d_S/2}\leq \#\mathrm{Extr}(u^{\eps})\leq c(\lambda^{\eps})^{d_S/2}
\end{equation}
with some constant $c>1$ independent of $\lambda_0$, $\mathbf a$ and $\eps$.
\end{theorem}

We will prove Theorem \ref{thm5} in Section \ref{sec5}.



\section{Proof of Theorem \ref{thm4}}\label{sec4}

We begin by introducing some notations.

An equivalent formulation of Proposition \ref{thm31}-(c) is as follows: for any $w\in W_*$ with $|w|=m\geq m_0$ and each $i\in S$,
\begin{equation*}
\tau^D(u\circ F_{wi})=P_{\lambda_{m+1}}^i\tau^D(u\circ F_w),
\end{equation*}
where $P_{\alpha}^i$ are defined for real $\alpha\notin\{2,5,6\}$ by
\begin{equation}\label{defP}
\begin{aligned}
&P_{\alpha}^1=\frac{1}{(2-\alpha)(5-\alpha)}
\begin{pmatrix}
(2-\alpha)(5-\alpha)&0&0\\
4-\alpha&4-\alpha&2\\
4-\alpha&2&4-\alpha
\end{pmatrix},\\
&P_{\alpha}^2=JP_{\alpha}^1J^{-1},\quad
P_{\alpha}^3=(J)^2P_{\alpha}^1(J^{-1})^2,\quad
\text{ with }J=
\begin{pmatrix}
0&0&1\\
1&0&0\\
0&1&0
\end{pmatrix}.
\end{aligned}
\end{equation}

Define $\mathcal{P}_{\alpha}^i:\mathbb{R}\mathrm{P}^2\to\mathbb{R}\mathrm{P}^2$ by $\mathcal{P}_{\alpha}^i=\pi\circ P_{\alpha}^i\circ\pi^{-1}$. Then, $\tau(u\circ F_{wi})=\mathcal{P}_{\lambda_{m+1}}^i\tau(u\circ F_w)$ for $w\in W_*$ with $|w|=m\geq m_0$ and $i\in S$, where $\tau=\pi\circ\tau^D$.

For simplicity, in this section we write $\lambda_0=\lambda_1(5-\lambda_1)=\psi^{-1}(\lambda)$ when $0<\lambda<\lambda_1^D$, so that $\lambda_{m}=\psi^{-1}(5^{-m}\lambda)$ for each $m\geq 0$. Moreover, by Remark \ref{re33}-(2), $\lambda_0\in(0,6)$ and $\lambda_m\in(0,2)$ for
$m\geq 1$.
Then, it follows directly from \eqref{defT} that $T_\lambda^i=P_{\lambda_1}^i=P_{\psi^{-1}(5^{-1}\lambda)}^i$, and for each $m\geq 0$ we have $T_{5^{-m}\lambda}^i=P_{\psi^{-1}(5^{-1}5^{-m}\lambda)}^i=P_{\lambda_{m+1}}^i$. Hence, $\mathcal T_{5^{-m}\lambda}^i=\mathcal P_{\lambda_{m+1}}^i$.

For $u\in E(\lambda)\setminus\{0\}$ with $0<\lambda<\lambda_1^D$ and $p_{ij}\in V_1\setminus V_0,\ ij\in S_1$, define
\begin{equation*}
\mathbf{M}_{\lambda,ij}=\Big\{\mathbf{a}\in\mathbb{R}^3\setminus\{\mathbf{0}\}: (\tau_\lambda^D)^{-1}(\mathbf{a})=u,\text{ with }u(p_{ij})>\max_{k\in S}u(p_k)\text{ or }u(p_{ij})<\min_{k\in S}u(p_k)\Big\}.
\end{equation*}
Then, $\mathbf{M}_{\lambda,ij}$ is a cone, and we write $\mathcal{M}_{\lambda,ij}=\pi(\mathbf{M}_{\lambda,ij})$.

Now we give a brief summary of the notations we defined on $\mathbb{R}\mathrm{P}^2$. For $i\in S$ and $ij\in S_1$,

(1). $\mathcal{P}_{\alpha}^i$ is the projective transformation induced by $P_{\alpha}^i$, where $\alpha\neq\{2,5,6\}$;

(2). $\mathcal{D}_\alpha$ defined in (\ref{defmd}) with $\alpha\in(0,6)$ is the open equilateral triangle centered at the origin $\boldsymbol{\theta}$ with vertices $\boldsymbol{\zeta}_{\alpha,23}$, $\boldsymbol{\zeta}_{\alpha,31}$ and $ \boldsymbol{\zeta}_{\alpha,12}$, which can be disjointly partitioned into three smaller open triangles $\mathcal G_{\alpha,i}$, three open line segments $\mathcal L_{\alpha,ij}$, and a single point $\boldsymbol{\theta}$;

furthermore, for $u\in E(\lambda)\setminus\{0\}$ with $0<\lambda<\lambda_1^D$,

(3). $\mathcal{T}_\lambda^i$ is the projective transformation induced by $T_\lambda^i$, satisfying $\mathcal{T}_{5^{-m}\lambda}^i=\mathcal{P}_{\lambda_{m+1}}^i$ for $m\geq 0$;

(4). $\mathcal{A}_\lambda$ is the region of $\tau(u)$ such that $u$ has at least one extreme set;

(5). $\mathcal{B}_\lambda$ is the region of $\tau(u)$ such that $u$ has exactly one extreme set that intersects $V_1$;

(6). $(\mathcal{T}_\lambda^i)^{-1}\mathcal{A}_{5^{-1}\lambda}$ is the region of $\tau(u)$ such that $u$ has at least one extreme set inside $F_i\mathcal{SG}$ (where ``inside'' means the extreme set does not intersect $F_iV_0$);

(7). $\mathcal{C}_\lambda$ is the region of $\tau(u)$ such that $(\mathrm{d}u)_{p_1},(\mathrm{d}u)_{p_2}$ and $(\mathrm{d}u)_{p_3}$ are either all positive or all negative;

(8). $\mathcal{M}_{\lambda,ij}$ is the region of $\tau(u)$ such that $u(p_{ij})$ is strictly greater (or less) than each of $u(p_1),u(p_2)$ and $u(p_3)$.

We divide the proof into two steps. In the first step we will show that $\mathcal C_\lambda=\mathcal{D}_{\lambda_0}$ and  $(\mathcal{T}_{\lambda}^i)^{-1}(\mathcal{C}_{5^{-1}\lambda})=(\mathcal{P}_{\lambda_1}^i)^{-1}(\mathcal{D}_{\lambda_1})=\mathcal{G}_{\lambda_0,i}$; and in the second step we will show that $\mathcal{A}_\lambda=\mathcal{C}_\lambda$ and $\mathcal{B}_\lambda=\big(\bigcup_{ij\in S_1}\mathcal{L}_{\lambda_0,ij}\big)\cup\{\boldsymbol{\theta}\}$, then complete the proof of Theorem \ref{thm4}.

\subsection{The first step} We start with two basic properties of $\mathcal P_\alpha^i$.

\begin{lemma}\label{pro41}
$\mathcal P_\alpha^i$ is a continuous bijection on $\mathbb{R}\mathrm{P}^2$ for any real $\alpha\notin\{2,5,6\}$ and $i\in S$, having the following explicit expression:
\begin{equation}\label{defpiP}
\begin{aligned}
&\mathcal{P}_{\alpha}^1(\boldsymbol{\xi})=
\begin{cases}
\Bigg(\mathbf{c}_{\alpha}^\mathrm{t}
\begin{pmatrix}
1\\ \boldsymbol{\xi}
\end{pmatrix}\Bigg)^{-1}
\Bigg(R_{\alpha}^\mathrm{t}
\begin{pmatrix}
1\\ \boldsymbol{\xi}
\end{pmatrix}\Bigg),\quad &\text{if \ } \mathbf{c}_{\alpha}^\mathrm{t}
\begin{pmatrix}
1\\ \boldsymbol{\xi}
\end{pmatrix}\neq 0\ ;\\
\Bigg[R_{\alpha}^\mathrm{t}
\begin{pmatrix}
1\\ \boldsymbol{\xi}
\end{pmatrix}\Bigg]_\infty, \quad &\text{if \ } \mathbf{c}_{\alpha}^\mathrm{t}
\begin{pmatrix}
1\\ \boldsymbol{\xi}
\end{pmatrix}=0\ ,\\
\end{cases}\\
&\mathcal{P}_{\alpha}^2=\mathcal{J}\circ\mathcal{P}_{\alpha}^1\circ\mathcal{J}^{-1},\quad \mathcal{P}_{\alpha}^3=\mathcal{J}^2\circ\mathcal{P}_{\alpha}^1\circ(\mathcal{J}^{-1})^2,
\end{aligned}
\end{equation}
with the convention $\begin{pmatrix}1\\ [\boldsymbol{\eta}]_\infty\end{pmatrix}
:=\begin{pmatrix}0\\ \boldsymbol{\eta}\end{pmatrix}$, and
\begin{equation*}
\begin{aligned}
&\mathbf{c}_{\alpha}^\mathrm{t}=\big((5-\alpha)(6-\alpha),\ 0,\ 2(2-\alpha)(6-\alpha)\big),\\ &R_{\alpha}^\mathrm{t}=\begin{pmatrix}
0&3(2-\alpha)&0\\
-\alpha(5-\alpha)&0&(2-\alpha)(9-2\alpha)
\end{pmatrix},
\end{aligned}
\end{equation*}
where $\mathcal{J}:\mathbb{R}\mathrm{P}^2\to\mathbb{R}\mathrm{P}^2$ is defined by
\begin{equation}\label{defpiJ}
\mathcal{J}(\boldsymbol{\xi})=
\begin{cases}
G\boldsymbol{\xi},\ \qquad\text{if }\boldsymbol{\xi}\in\mathbb{R}^2;\\
[G\boldsymbol{\eta}]_\infty,\quad\text{if }\boldsymbol{\xi}=[\boldsymbol{\eta}]_\infty\in L_\infty,
\end{cases}
\quad G=\frac{1}{2}
\begin{pmatrix}
-1&-\sqrt{3}\\
\sqrt{3}&-1
\end{pmatrix}.
\end{equation}
\end{lemma}

\begin{proof}
A direct calculation shows that the matrix $P_\alpha^i$ has eigenvalues $1,\frac{6-\alpha}{(2-\alpha)(5-\alpha)},\frac{1}{5-\alpha}$ for any $i\in S$, so that $P_\alpha^i$ is invertible for $\alpha\notin\{2,5,6\}$. It follows that $\mathcal P_\alpha^i=\pi\circ P_\alpha^i\circ\pi^{-1}$ is a continuous bijection.

From (\ref{defpi}), for each $\boldsymbol{\xi}\in\mathbb{R}\mathrm{P}^2$ we can find $\mathbf{x}\in\pi^{-1}(\boldsymbol{\xi})$ by
\begin{equation}\label{eq4.4}
\mathbf{x}=
\begin{cases}
(\mathbf{1},Q)^{-\mathrm{t}}
\begin{pmatrix}
1\\ \boldsymbol{\xi}
\end{pmatrix},\quad\text{if}\quad \boldsymbol{\xi}\in\mathbb{R}^2;
\\
(\mathbf{1},Q)^{-\mathrm{t}}
\begin{pmatrix}
0\\ \boldsymbol{\eta}
\end{pmatrix},\quad\text{if}\quad \boldsymbol{\xi}=[\boldsymbol{\eta}]_\infty\in L_\infty,
\end{cases}
\end{equation}
where $(\mathbf{1},Q)^{-\mathrm{t}}=((\mathbf{1},Q)^\mathrm{t})^{-1}$, $L_\infty=\mathbb{R}\mathrm{P}^2\setminus\mathbb{R}^2$. For consistency in later context, we set 
$
\begin{pmatrix}
1\\ [\boldsymbol{\eta}]_\infty
\end{pmatrix}
=
\begin{pmatrix}
0\\ \boldsymbol{\eta}
\end{pmatrix}
$ for $[\boldsymbol{\eta}]_\infty\in L_\infty$.

Then, for $\boldsymbol{\xi}\in \mathbb{R}\mathrm{P}^2$, we see that
\begin{equation*}
\mathcal{P}_{\alpha}^1(\boldsymbol{\xi})=\pi\circ P_{\alpha}^1\circ\pi^{-1}(\boldsymbol{\xi})
=\pi\Bigg(P_{\alpha}^1(\mathbf{1},Q)^{-\mathrm{t}}
\begin{pmatrix}
1\\ \boldsymbol{\xi}
\end{pmatrix}\Bigg).
\end{equation*}
Recalling (\ref{defpi}) for the definition of $\pi$ and noticing that for $\mathbf y\in \mathbb R^3\setminus \{\mathbf 0\}$, 
$(1,Q)^\mathrm t\mathbf y=\begin{pmatrix}
1^\mathrm t\mathbf y\\ Q^\mathrm t\mathbf{y}
\end{pmatrix}$,
we get (\ref{defpiP}) since
\begin{equation*}
\begin{aligned}
(\mathbf{1},Q)^\mathrm{t}P_{\alpha}^1(\mathbf{1},Q)^{-\mathrm{t}}\begin{pmatrix}
1\\ \boldsymbol{\xi}
\end{pmatrix}
=\frac{1}{3(2-\alpha)(5-\alpha)}(\mathbf{c}_{\alpha},R_{\alpha})^\mathrm{t}\begin{pmatrix}
1\\ \boldsymbol{\xi}
\end{pmatrix}.
\end{aligned}
\end{equation*}

Further, define $\mathcal{J}=\pi\circ J\circ\pi^{-1}$. Since
\begin{equation*}
(\mathbf{1},Q)^\mathrm{t}J(\mathbf{1},Q)^{-\mathrm{t}}=\frac{1}{2}
\begin{pmatrix}
2&0&0\\
0&-1&-\sqrt{3}\\
0&\sqrt{3}&-1
\end{pmatrix},
\end{equation*}
we have (\ref{defpiJ}) holds and
\begin{equation*}
\begin{aligned}
&\mathcal{P}_{\alpha}^2=\pi\circ P_{\alpha}^2\circ\pi^{-1}=\pi\circ J\circ P_{\alpha}^1\circ J^{-1}\circ\pi^{-1}=\mathcal{J}\circ\mathcal{P}_{\alpha}^1\circ\mathcal{J}^{-1},\\ &\mathcal{P}_{\alpha}^3=\pi\circ J^2\circ P_{\alpha}^1\circ (J^{-1})^2\circ\pi^{-1}=\mathcal{J}^2\circ\mathcal{P}_{\alpha}^1\circ(\mathcal{J}^{-1})^2.
\end{aligned}
\end{equation*}
\end{proof}

\begin{lemma}\label{pro41r}
(a). For $\alpha\in(0,2)$,
\begin{equation*}
(\mathcal P_\alpha^1)^{-1}(\mathcal D_\alpha)=\mathcal{G}_{\Phi(\alpha),1}=\Big\{(\xi^{(1)},\xi^{(2)})^\mathrm{t}:
\frac{\Phi(\alpha)}{2(6-\Phi(\alpha))}>\xi^{(2)}>\frac{1}{\sqrt{3}}|\xi^{(1)}|\Big\};
\end{equation*}

(b). For $\alpha\in(3,5)$,
\begin{equation*}
\begin{aligned}
(\mathcal P_\alpha^1)^{-1}(\mathcal D_\alpha)&=\Big\{(\xi^{(1)},\xi^{(2)})^\mathrm{t}:
\xi^{(2)}>\frac{\Phi(\alpha)}{2(6-\Phi(\alpha))}\text{ and }\xi^{(2)}>\frac{1}{\sqrt{3}}|\xi^{(1)}|\Big\}\cup\\
&\Big\{(\xi^{(1)},\xi^{(2)})^\mathrm{t}:
\xi^{(2)}<-\frac{1}{\sqrt{3}}|\xi^{(1)}|\Big\}\cup\Big\{[(\eta^{(1)},\eta^{(2)})^\mathrm{t}]_\infty:
|\eta^{(2)}|>\frac{1}{\sqrt{3}}|\eta^{(1)}|\Big\}.
\end{aligned}
\end{equation*}

(c). For $\alpha=3$,
\begin{equation*}
(\mathcal P_\alpha^1)^{-1}(\mathcal D_\alpha)=\Big\{(\xi^{(1)},\xi^{(2)})^\mathrm{t}:
\xi^{(2)}<-\frac{1}{\sqrt{3}}|\xi^{(1)}|\Big\}.
\end{equation*}

Moreover, $(\mathcal P_\alpha^2)^{-1}(\mathcal D_\alpha)=\mathcal J\big((\mathcal P_\alpha^1)^{-1}(\mathcal D_\alpha)\big),\ (\mathcal P_\alpha^3)^{-1}(\mathcal D_\alpha)=\mathcal J^2\big((\mathcal P_\alpha^1)^{-1}(\mathcal D_\alpha)\big)$.
\end{lemma}

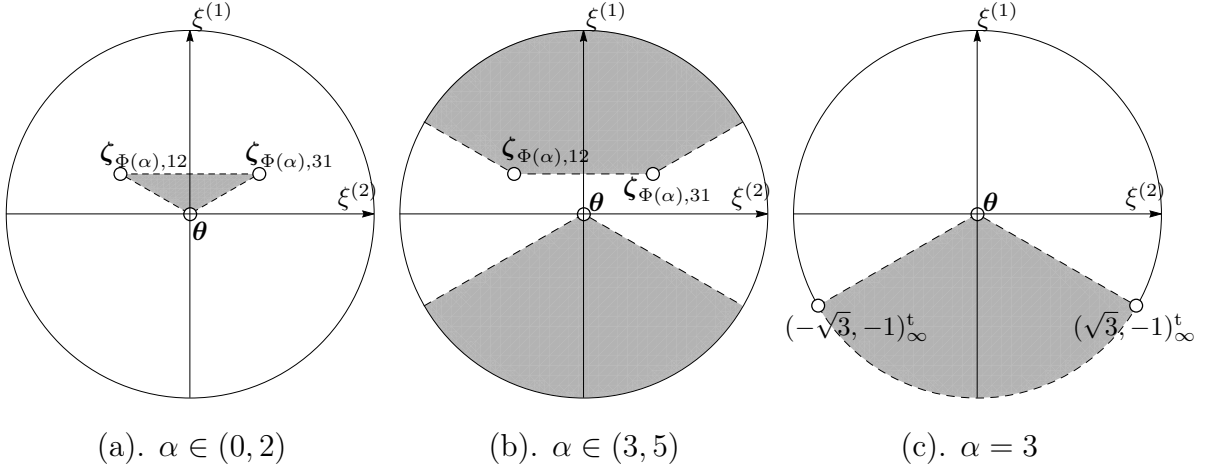
\begin{figure}[h]
\begin{center}
\includegraphics[width=4.9cm]{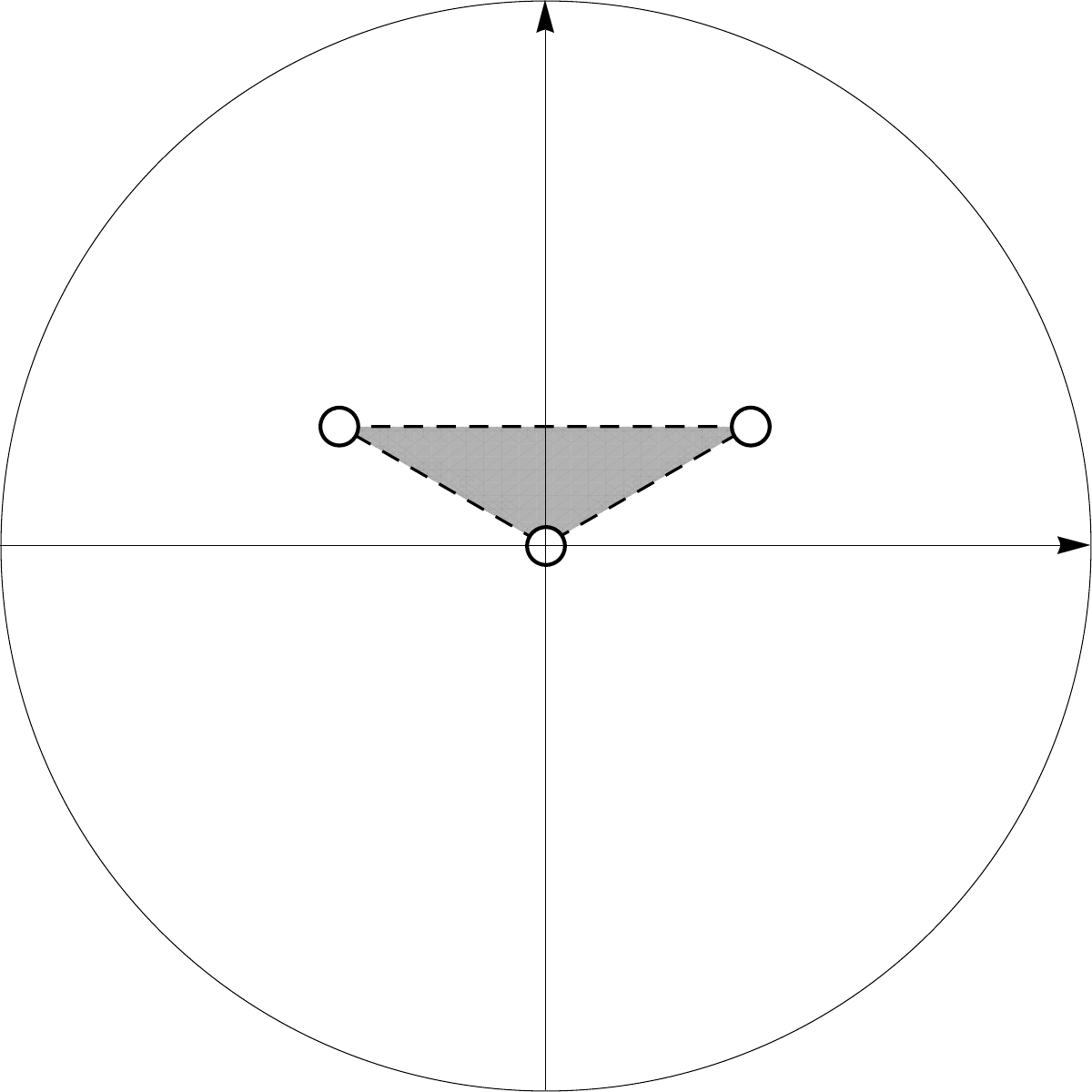}\hspace{0.2cm}
\includegraphics[width=4.9cm]{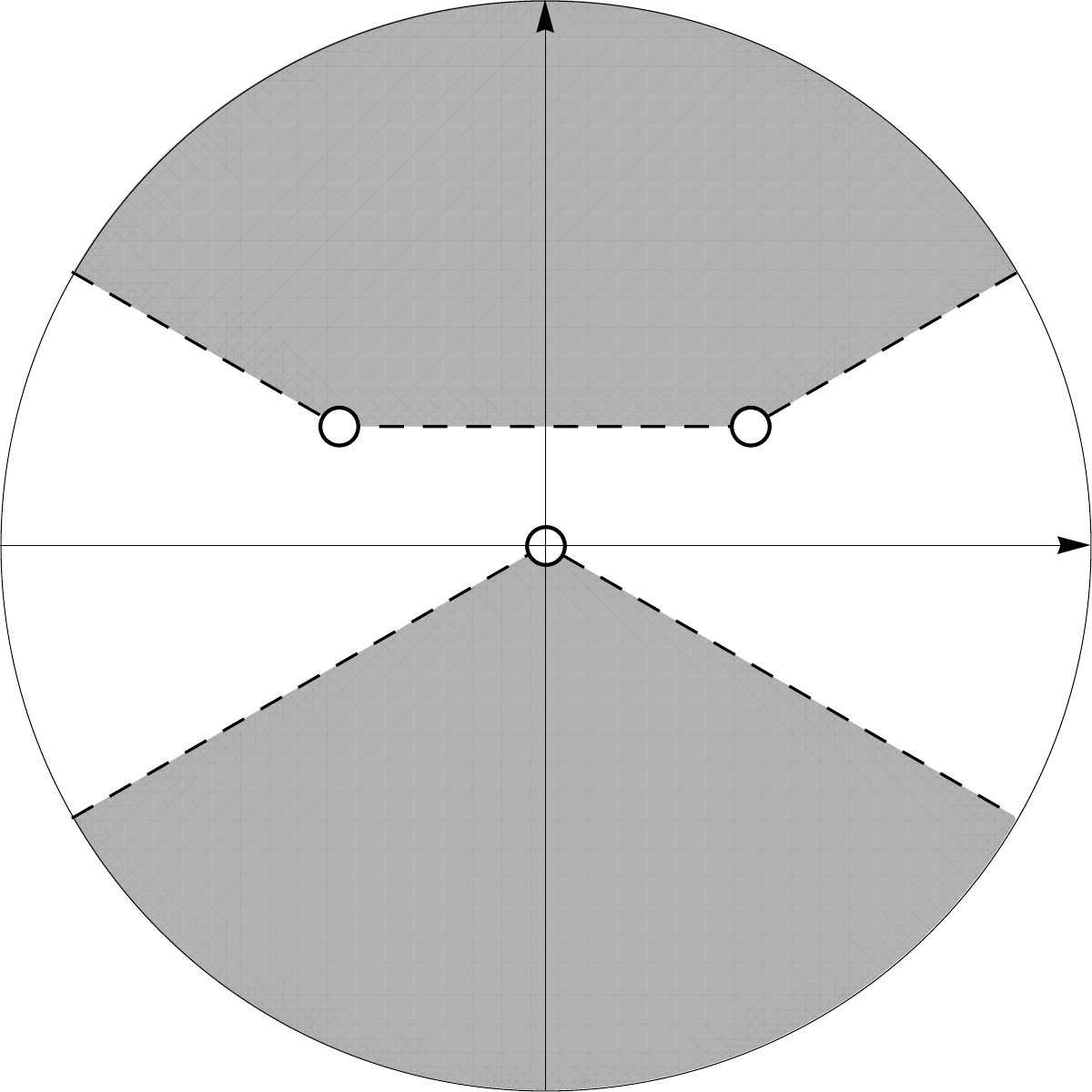}\hspace{0.2cm}
\includegraphics[width=4.9cm]{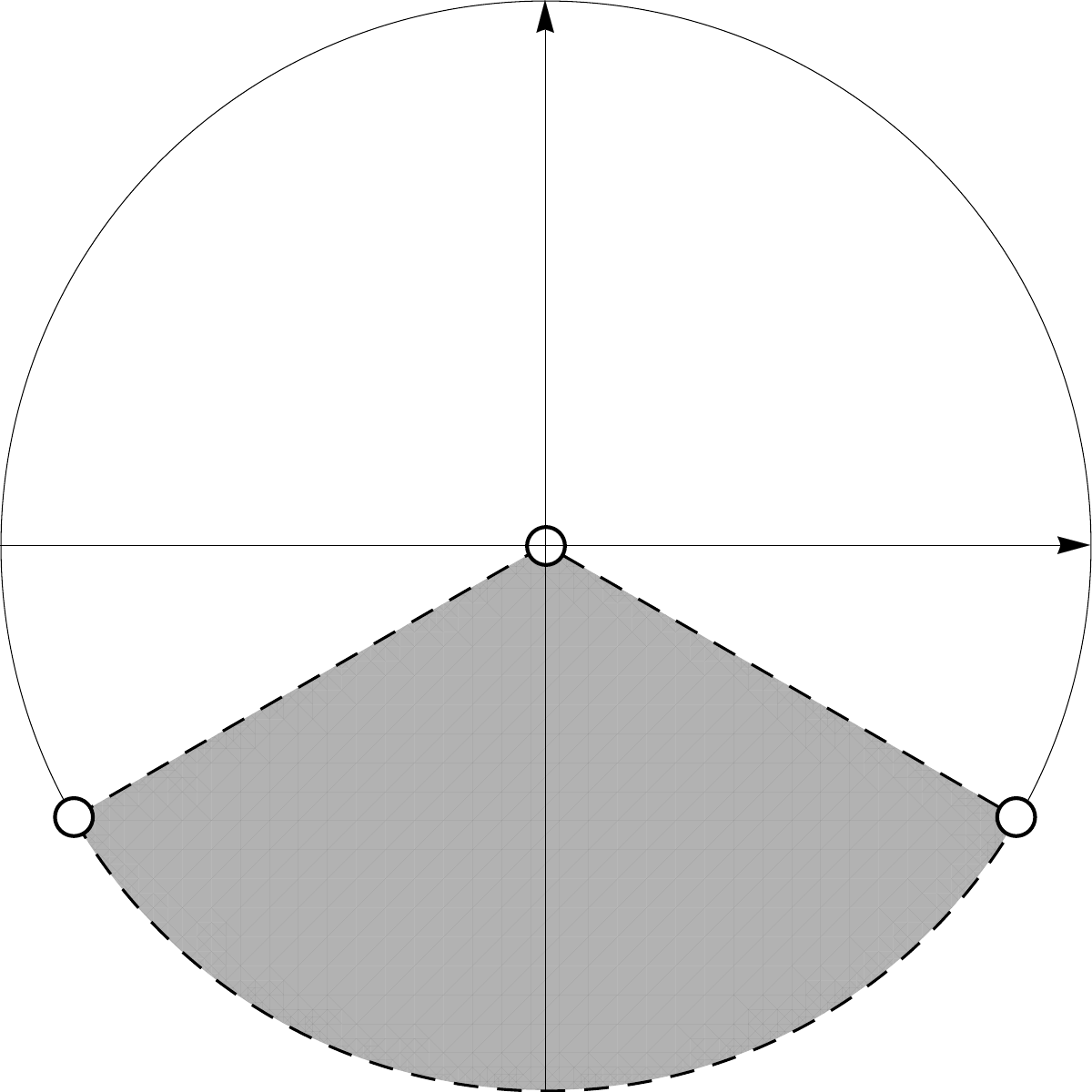}
\vspace{1cm}\\
\caption{$(\mathcal P_\alpha^1)^{-1}(\mathcal D_\alpha)$ (the shaded region), and the outer circle represents $L_\infty$, the line at infinity.}
\label{figurepalpha}
\end{center}
\begin{picture}(0,0) \thicklines
\put(147,221){\footnotesize$\xi^{(1)}$}
\put(-150,221){\footnotesize$\xi^{(1)}$}
\put(-1,221){\footnotesize$\xi^{(1)}$}
\put(203,153){\footnotesize$\xi^{(2)}$}
\put(-95,153){\footnotesize$\xi^{(2)}$}
\put(54,153){\footnotesize$\xi^{(2)}$}
\put(149,151){\footnotesize$\boldsymbol{\theta}$}
\put(-150,140){\footnotesize$\boldsymbol{\theta}$}
\put(0,151){\footnotesize$\boldsymbol{\theta}$}
\put(13,157){\footnotesize$\boldsymbol{\zeta}_{\Phi(\alpha),31}$}
\put(-130,170){\footnotesize$\boldsymbol{\zeta}_{\Phi(\alpha),31}$}
\put(-33,172){\footnotesize$\boldsymbol{\zeta}_{\Phi(\alpha),12}$}
\put(-185,170){\footnotesize$\boldsymbol{\zeta}_{\Phi(\alpha),12}$}
\put(183,104){\footnotesize$(\sqrt{3},-1)^\mathrm{t}_\infty$}
\put(74,104){\footnotesize$(-\sqrt{3},-1)^\mathrm{t}_\infty$}
\put(-38,58){(b). $\alpha\in(3,5)$}
\put(-186,58){(a). $\alpha\in(0,2)$}
\put(118,58){(c). $\alpha=3$}
\end{picture}
\end{figure}

\begin{proof}
Since $\mathcal P_\alpha^1$ is induced by the invertible linear transformation $P_\alpha^1$, it maps (projective) triangles to (projective) triangles. Noticing that on $\mathbb R\mathrm P^2$ we need three vertices and an interior point to determine a projective triangle, to determine $(\mathcal P_\alpha^1)^{-1}(\mathcal D_\alpha)$, we only need to consider the three vertices of $\mathcal D_\alpha$, together with $\boldsymbol{\theta}$, one of its interior points.

By Lemma \ref{pro41} and a direct computation, for $\alpha\neq 3$,
\begin{equation}\label{eqpalpha}
\begin{aligned}
&\mathcal{P}_{\alpha}^1(\boldsymbol{\theta})=\mathcal{P}_{\alpha}^1\Bigg(\begin{pmatrix}0\\0\end{pmatrix}\Bigg)=\frac{\alpha}{6-\alpha}\begin{pmatrix}0\\-1\end{pmatrix}=\boldsymbol{\zeta}_{\alpha,23},\\
&\mathcal{P}_{\alpha}^1(\boldsymbol{\zeta}_{\Phi(\alpha),31})=\mathcal{P}_{\alpha}^1\Bigg(\frac{\alpha(5-\alpha)}{2(2-\alpha)(3-\alpha)}\begin{pmatrix}\sqrt{3}\\1\end{pmatrix}\Bigg)=\frac{\alpha}{2(6-\alpha)}\begin{pmatrix}\sqrt{3}\\1\end{pmatrix}=\boldsymbol{\zeta}_{\alpha,31},\\
&\mathcal{P}_{\alpha}^1(\boldsymbol{\zeta}_{\Phi(\alpha),12})=\mathcal{P}_{\alpha}^1\Bigg(\frac{\alpha(5-\alpha)}{2(2-\alpha)(3-\alpha)}\begin{pmatrix}-\sqrt{3}\\1\end{pmatrix}\Bigg)=\frac{\alpha}{2(6-\alpha)}\begin{pmatrix}-\sqrt{3}\\1\end{pmatrix}=\boldsymbol{\zeta}_{\alpha,12},\\
&\mathcal{P}_{\alpha}^1\Bigg(\frac{2(3-\alpha)}{9-2\alpha}\cdot\frac{\Phi(\alpha)}{2(6-\Phi(\alpha))}\begin{pmatrix}0\\1\end{pmatrix}\Bigg)=\begin{pmatrix}0\\0\end{pmatrix}=\boldsymbol{\theta},
\end{aligned}
\end{equation}
where $\frac{2(3-\alpha)}{9-2\alpha}\in(0,1)$ for $\alpha\in(0,2)$, and $\frac{2(3-\alpha)}{9-2\alpha}\notin(0,1)$ for $\alpha\in(3,5)$ ($\mathcal{P}_{9/2}^1([(0,1)^\mathrm t]_\infty)=\boldsymbol{\theta}$). So (a) and (b) follow immediately.

For $\alpha=3$, noticing that $\Phi(\alpha)=6\notin(0,6)$, we have
\begin{equation}\label{eqpalpha3}
\begin{aligned}
&\mathcal{P}_{\alpha}^1(\boldsymbol{\theta})=\mathcal{P}_{\alpha}^1\Bigg(\begin{pmatrix}0\\0\end{pmatrix}\Bigg)=\frac{\alpha}{6-\alpha}\begin{pmatrix}0\\-1\end{pmatrix}=\boldsymbol{\zeta}_{\alpha,23},\\
&\mathcal{P}_{\alpha}^1\Bigg(\bigg[\begin{pmatrix}\sqrt{3}\\1\end{pmatrix}\bigg]_\infty\Bigg)=\frac{\alpha}{2(6-\alpha)}\begin{pmatrix}\sqrt{3}\\1\end{pmatrix}=\boldsymbol{\zeta}_{\alpha,13},\\
&\mathcal{P}_{\alpha}^1\Bigg(\bigg[\begin{pmatrix}-\sqrt{3}\\1\end{pmatrix}\bigg]_\infty\Bigg)=\frac{\alpha}{2(6-\alpha)}\begin{pmatrix}-\sqrt{3}\\1\end{pmatrix}=\boldsymbol{\zeta}_{\alpha,12},\\
&\mathcal{P}_{\alpha}^1\Bigg(\frac{\alpha(5-\alpha)}{(2-\alpha)(9-2\alpha)}\begin{pmatrix}0\\1\end{pmatrix}\Bigg)=\begin{pmatrix}0\\0\end{pmatrix}=\boldsymbol{\theta},
\end{aligned}
\end{equation}
where $\frac{\alpha(5-\alpha)}{(2-\alpha)(9-2\alpha)}=-2$, gives (c). 

Finally, noticing that $\mathcal J^{-1}(\mathcal D_\alpha)=\mathcal D_\alpha$, we see that $(\mathcal P_\alpha^2)^{-1}(\mathcal D_\alpha)=\mathcal J\big((\mathcal P_\alpha^1)^{-1}(\mathcal D_\alpha)\big)$ and $(\mathcal P_\alpha^3)^{-1}(\mathcal D_\alpha)=\mathcal J^2\big((\mathcal P_\alpha^1)^{-1}(\mathcal D_\alpha)\big)$.
\end{proof}

\noindent\textbf{Remark.} Note that when $\alpha\in (3,5)$, the set $(\mathcal P_\alpha^1)^{-1}(\mathcal D_\alpha)$ contains a part of the line at infinity of $\mathbb{R}\mathrm{P}^2$, whereas when $\alpha=3$, it does not. See Figure \ref{figurepalpha} for details of $(\mathcal P_\alpha^1)^{-1}(\mathcal D_\alpha)$ when $\alpha\in (0,2)$, $\alpha\in (3,5)$ or $\alpha=3$.

We also need the following observation for functions in $E(\lambda)\setminus\{0\}$.

\begin{lemma}\label{pro43}
For $u\in E(\lambda)\setminus\{0\}$ with $0<\lambda<\lambda_1^D$ and distinct $i,j,k\in S$, $(\mathrm du)_{p_i}=0$ if and only if $(4-\lambda_0)u(p_i)-2(u(p_j)+u(p_k))=0$.
\end{lemma}

\begin{proof}
If $(4-\lambda_0)u(p_1)-2(u(p_2)+u(p_3))=0$, we have $(4-\lambda_1)u(p_1)-2(u(F_1p_2)+u(F_1p_3))=0$, since by (\ref{eqdeci}), 
\begin{equation*}
\begin{aligned}
2(u(F_1p_2)+&u(F_1p_3))=2(u(p_{12})+u(p_{13}))\\
&=\frac{2}{(2-\lambda_1)(5-\lambda_1)}\Big(2(4-\lambda_1)u(p_1)+(6-\lambda_1)(u(p_2)+u(p_3))\Big)\\
&=\frac{4(4-\lambda_1)+(6-\lambda_1)(4-\lambda_0)}{(2-\lambda_1)(5-\lambda_1)}u(p_1)\\
&=(4-\lambda_1)u(p_1),
\end{aligned}
\end{equation*}
noticing that $\lambda_0=\lambda_1(5-\lambda_1)$ and $\lambda_1\in(0,2)$. Similarly, for each $m\geq 1$, if  $(4-\lambda_{m})u(p_1)-2(u(F_1^{m}p_2)+u(F_1^{m}p_3))=0$, then $(4-\lambda_{m+1})u(p_1)-2(u(F_1^{m+1}p_2)+u(F_1^{m+1}p_3))=0$. 

So by induction we have $(4-\lambda_m)u(p_1)-2(u(F_1^mp_2)+u(F_1^mp_3))=0$ for any $m\geq 0$, and this gives
\begin{equation*}
\begin{aligned}
(\mathrm{d}u)_{p_1}&=\lim_{m\to\infty}\Big(\frac{5}{3}\Big)^m\big(2u(p_1)-u(F_1^mp_2)-u(F_1^mp_3)\big)\\
&=\lim_{m\to\infty}\Big(\frac{5}{3}\Big)^m\frac{\lambda_m}{2}u(p_1)=0.
\end{aligned}
\end{equation*}

Therefore, we have
\begin{equation*}
\begin{aligned}
&\big\{(x^{(1)},x^{(2)},x^{(3)})^\mathrm{t}\in\mathbb{R}^3:(4-\lambda_0)x^{(1)}-2(x^{(2)}+x^{(3)})=0\big\}\\
=&\tau_\lambda^D\Big(\big\{u\in E(\lambda):(4-\lambda_0)u(p_1)-2(u(p_2)+u(p_3))=0\big\}\Big)\\
\subset &\tau_\lambda^D\Big(\{u\in E(\lambda):(\mathrm{d}u)_{p_1}=0\}\Big)\\
=&\tau_\lambda^D\circ(\tau_\lambda^N)^{-1}\Big(\{(x^{(1)},x^{(2)},x^{(3)})^\mathrm{t}\in\mathbb{R}^3:x^{(1)}=0\}\Big).
\end{aligned}
\end{equation*}
Since by Lemma \ref{thm2.5}, $\tau_\lambda^D,\ \tau_\lambda^N$ are both linear and invertible, the inclusion  ``$\subset$'' above can be replaced by ``$=$'', which gives $(\mathrm du)_{p_1}=0$ if and only if $(4-\lambda_0)u(p_1)-2(u(p_2)+u(p_3))=0$.

The case for other distinct $i,j,k\in S$ follows by symmetry.
\end{proof}

The next proposition is the main result in the first step.

\begin{proposition}\label{lem44}
Suppose $0<\lambda<\lambda_1^D$, and $m\geq 0$ is an integer. Then,
\begin{equation}\label{exC}
\begin{aligned}
&\mathcal{C}_{5^{-m}\lambda}=\mathcal{D}_{\lambda_m}=\{(\xi^{(1)},\xi^{(2)})^\mathrm{t}\in\mathbb{R}\mathrm{P}^2:\\
&\quad\quad\xi^{(2)}<\frac{\lambda_m}{2(6-\lambda_m)},\ \xi^{(2)}>-\sqrt{3}\xi^{(1)}-\frac{\lambda_m}{6-\lambda_m},\ \xi^{(2)}>\sqrt{3}\xi^{(1)}-\frac{\lambda_m}{6-\lambda_m}\},
\end{aligned}
\end{equation}
\begin{equation}\label{exM}
\begin{aligned}
&\mathcal{M}_{5^{-m}\lambda,23}=\{(\xi^{(1)},\xi^{(2)})^\mathrm{t}\in\mathbb{R}\mathrm{P}^2:\\
&\quad\quad\xi^{(2)}<\frac{3-\lambda_{m+1}}{6-\lambda_{m+1}}\cdot\frac{\lambda_m}{2(6-\lambda_m)},\ \xi^{(2)}>-\frac{5-\lambda_{m+1}}{3-\lambda_{m+1}}\cdot\sqrt{3}\xi^{(1)}-\frac{\lambda_m}{6-\lambda_m},\\ &\quad\quad\xi^{(2)}>\frac{5-\lambda_{m+1}}{3-\lambda_{m+1}}\cdot\sqrt{3}\xi^{(1)}-\frac{\lambda_m}{6-\lambda_m}\},\\
&\mathcal{M}_{5^{-m}\lambda,31}=\mathcal{J}(\mathcal{M}_{5^{-m}\lambda,23}),\quad \mathcal{M}_{5^{-m}\lambda,12}=\mathcal{J}^2(\mathcal{M}_{5^{-m}\lambda,23}).
\end{aligned}
\end{equation}

Moreover, we have $(\mathcal{T}_{5^{-m}\lambda}^i)^{-1}(\mathcal{C}_{5^{-m-1}\lambda})=(\mathcal{P}_{\lambda_{m+1}}^i)^{-1}(\mathcal{D}_{\lambda_{m+1}})=\mathcal{G}_{\lambda_m,i}$ for $i\in S$, and
\begin{equation}\label{decom}
\begin{aligned}
\mathcal{C}_{5^{-m}\lambda}=\big(\bigcup_{i\in S}(\mathcal{T}_{5^{-m}\lambda}^i)^{-1}(\mathcal{C}_{5^{-m-1}\lambda})\big)\cup\big(\bigcup_{ij\in S_1}\mathcal{L}_{\psi^{-1}(5^{-m}\lambda),ij}\big)\cup\{\boldsymbol{\theta}\}
\end{aligned}
\end{equation}
is a disjoint decomposition.
In addition, $\mathcal{L}_{\psi^{-1}(5^{-m}\lambda),ij}\subset\mathcal{M}_{5^{-m}\lambda,ij}\subset\mathcal{C}_{5^{-m}\lambda}$ and $\big(\partial\mathcal{M}_{5^{-m}\lambda,ij}\big)\setminus\{\boldsymbol{\zeta}_{\psi^{-1}(5^{-m}\lambda),ij}\}\subset\mathcal{C}_{5^{-m}\lambda}$ for any $ij\in S_1$.
\end{proposition}

\begin{proof} It suffices to prove the $m=0$ case.

First we prove \eqref{exC}, i.e. $\mathcal{C}_{\lambda}=\mathcal{D}_{\lambda_0}$. Noticing that for any $\mathbf{b}\in\mathbb{R}^3\setminus\{{\bf 0}\}$, by \eqref{eq4.4}, we have
\begin{equation}\label{pixxi}
\pi(\{\mathbf{x}\in\mathbb{R}^3\setminus\{{\bf 0}\}:\mathbf{b}^\mathrm{t}\mathbf{x}=0\})=\Big\{\boldsymbol{\xi}\in\mathbb{R}\mathrm{P}^2:\ \mathbf{b}^\mathrm{t}(\mathbf{1},Q)^{-\mathrm{t}}\begin{pmatrix}1\\ \boldsymbol{\xi}\end{pmatrix}=0\Big\},
\end{equation}
and a direct calculation gives
\begin{equation*}
(4-\lambda_0,-2,-2)(\mathbf{1},Q)^{-\mathrm{t}}=\frac{1}{3}\big(-\lambda_0,0,2(6-\lambda_0)\big).
\end{equation*}
So by Lemma \ref{pro43}, we have
\begin{equation*}
\begin{aligned}
&\pi\circ\tau_\lambda^D\circ(\tau_\lambda^N)^{-1}\Big(\{(x^{(1)},x^{(2)},x^{(3)})^\mathrm{t}\in\mathbb{R}^3\setminus\{{\bf 0}\}:x^{(1)}=0\}\Big)\\
=&\Big\{\boldsymbol{\xi}\in\mathbb{R}\mathrm{P}^2:\ \big(-\lambda_0,0,2(6-\lambda_0)\big)\begin{pmatrix}1\\ \boldsymbol{\xi}\end{pmatrix}=0\Big\}\\
=&\Big\{(\xi^{(1)},\xi^{(2)})^\mathrm{t}\in\mathbb{R}^2:\ \xi^{(2)}-\frac{\lambda_0}{2(6-\lambda_0)}=0\Big\}\cup\{[(1,0)^\mathrm{t}]_\infty\}.
\end{aligned}
\end{equation*}
Similarly, $(-2,4-\lambda_0,-2)(\mathbf{1},Q)^{-\mathrm{t}}=\frac{1}{3}\big(-\lambda_0,-\sqrt{3}(6-\lambda_0),-(6-\lambda_0)\big)$ gives
\begin{equation*}
\begin{aligned}
&\pi\circ\tau_\lambda^D\circ(\tau_\lambda^N)^{-1}\Big(\{(x^{(1)},x^{(2)},x^{(3)})^\mathrm{t}\in\mathbb{R}^3\setminus\{{\bf 0}\}:x^{(2)}=0\}\Big)\\
=&\{(\xi^{(1)},\xi^{(2)})^\mathrm{t}\in\mathbb{R}^2:\ \sqrt{3}\xi^{(1)}+\xi^{(2)}+\frac{\lambda_0}{6-\lambda_0}=0\}\cup\{[(1,-\sqrt{3})^\mathrm{t}]_\infty\},
\end{aligned}
\end{equation*}
and $(-2,-2,4-\lambda_0)(\mathbf{1},Q)^{-\mathrm{t}}=\frac{1}{3}\big(-\lambda_0,\sqrt{3}(6-\lambda_0),-(6-\lambda_0)\big)$ gives
\begin{equation*}
\begin{aligned}
&\pi\circ\tau_\lambda^D\circ(\tau_\lambda^N)^{-1}\Big(\{(x^{(1)},x^{(2)},x^{(3)})^\mathrm{t}\in\mathbb{R}^3\setminus\{{\bf 0}\}:x^{(3)}=0\}\Big)\\
=&\{(\xi^{(1)},\xi^{(2)})^\mathrm{t}\in\mathbb{R}^2:\ -\sqrt{3}\xi^{(1)}+\xi^{(2)}+\frac{\lambda_0}{6-\lambda_0}=0\}\cup\{[(1,\sqrt{3})^\mathrm{t}]_\infty\}.
\end{aligned}
\end{equation*}
Recall that $\mathbf{C}=\{(a_1,a_2,a_3)\in\mathbb{R}^3:\text{either }a_i>0\text{ for all }i, \text{ or }a_i<0\text{ for all }i\}$ and $\mathcal{C}_\lambda=\pi\circ\tau_\lambda^D\circ(\tau_\lambda^N)^{-1}(\mathbf{C})$. By the symmetry of $\mathcal{SG}$, we have $\boldsymbol{\theta}\in\mathcal{C}_\lambda$.  So the above calculation gives
$\mathcal{C}_{\lambda}=\mathcal{D}_{\lambda_0}$.

Next we prove \eqref{exM}. Noticing that for $i\in S$, by (\ref{eqdeci}) and (\ref{pixxi}), $\pi\circ\tau_\lambda^D\Big(\big\{u\in E(\lambda)\setminus\{0\}:u(p_{23})=u(p_i)\big\}\Big)=\Big\{\boldsymbol{\xi}\in\mathbb{R}\mathrm{P}^2:\ \mathbf{b}_i^\mathrm{t}(\mathbf{1},Q)^{-\mathrm{t}}\begin{pmatrix}1\\ \boldsymbol{\xi}\end{pmatrix}=0\Big\}$ with 
\begin{equation*}
\begin{aligned}
&\mathbf{b}_1^\mathrm{t}=\frac{1}{(2-\lambda_1)(5-\lambda_1)}\big(2,4-\lambda_1,4-\lambda_1\big)-(1,0,0),\\
&\mathbf{b}_2^\mathrm{t}=\frac{1}{(2-\lambda_1)(5-\lambda_1)}\big(2,4-\lambda_1,4-\lambda_1\big)-(0,1,0),\\
&\mathbf{b}_3^\mathrm{t}=\frac{1}{(2-\lambda_1)(5-\lambda_1)}\big(2,4-\lambda_1,4-\lambda_1\big)-(0,0,1).
\end{aligned}
\end{equation*}
A direct calculation gives
\begin{equation*}
\begin{aligned}
&\mathbf{b}_1^\mathrm{t}(\mathbf{1},Q)^{-\mathrm{t}}=\frac{1}{3}\big(\frac{\lambda_1}{2-\lambda_1},0,\frac{-2(6-\lambda_1)}{5-\lambda_1}\big),\\
&\mathbf{b}_2^\mathrm{t}(\mathbf{1},Q)^{-\mathrm{t}}=\frac{1}{3}\big(\frac{\lambda_1}{2-\lambda_1},\sqrt{3},\frac{3-\lambda_1}{5-\lambda_1}\big),\\
&\mathbf{b}_3^\mathrm{t}(\mathbf{1},Q)^{-\mathrm{t}}=\frac{1}{3}\big(\frac{\lambda_1}{2-\lambda_1},-\sqrt{3},\frac{3-\lambda_1}{5-\lambda_1}\big).
\end{aligned}
\end{equation*}
Similar to the calculation for $\mathcal{C}_{\lambda}=\mathcal{D}_{\lambda_0}$, we obtain \eqref{exM} for $\mathcal{M}_{\lambda,23}$, noticing that $\lambda_0=\Phi(\lambda_1)$ and $\boldsymbol{\theta}\in\mathcal{M}_{\lambda,23}$. In addition, by symmetry, $\mathcal{M}_{\lambda,31}=\mathcal J(\mathcal{M}_{\lambda,23}),\ \mathcal{M}_{\lambda,12}=\mathcal J^2(\mathcal{M}_{\lambda,23})$ follows from $\mathbf{M}_{\lambda,31}=J(\mathbf{M}_{\lambda,23}),\ \mathbf{M}_{\lambda,12}=J^2(\mathbf{M}_{\lambda,23})$.

Finally, noticing that $\mathcal{C}_{\lambda}=\mathcal{D}_{\lambda_0}$, $\mathcal{T}_{\lambda}^i=\mathcal{P}_{\lambda_{1}}^i$ and $\lambda_0=\Phi(\lambda_{1})$, by Lemma \ref{pro41r}-(a) we have $(\mathcal{T}_{\lambda}^i)^{-1}(\mathcal{C}_{5^{-1}\lambda})=(\mathcal{P}_{\lambda_{1}}^i)^{-1}(\mathcal{D}_{\lambda_{1}})=\mathcal{G}_{\lambda_0,i}$ for $i\in S$. Then we have (\ref{decom}) since
\begin{equation*}
\mathcal{D}_{\lambda_0}=\big(\bigcup_{i\in S}\mathcal{G}_{\lambda_0,i}\big)\cup\big(\bigcup_{ij\in S_1}\mathcal{L}_{\lambda_0,ij}\big)\cup\{\boldsymbol{\theta}\}.
\end{equation*}
is a disjoint decomposition. Moreover, $\mathcal{L}_{\psi^{-1}(\lambda),ij}\subset\mathcal{M}_{\lambda,ij}\subset\mathcal{C}_{\lambda}$ and $\big(\partial\mathcal{M}_{\lambda,ij}\big)\setminus\{\boldsymbol{\zeta}_{\psi^{-1}(\lambda),ij}\}\subset\mathcal{C}_{\lambda}$ is obvious.
\end{proof}

\subsection{The second step} We begin with some lemmas.

\begin{lemma}\label{matching} 
For $u\in\mathscr{D}$ and $ij\in S_1$, we have
\begin{equation*}
\big(\mathrm{d}(u\circ F_i)\big)_{p_j}+\big(\mathrm{d}(u\circ F_j)\big)_{p_i}=0.
\end{equation*}
\end{lemma}
\begin{proof}
By (\ref{defdeltam}) and (\ref{defdelta}), for distinct $i,j,k\in S$ we have
\begin{equation*}
\begin{aligned}
\Delta u(p_{ij})=\frac{3}{2}\lim_{m\to\infty}5^m&\big(u(F_iF_j^{m-1}p_k)+u(F_iF_j^{m-1}p_i)\\
&+u(F_jF_i^{m-1}p_k)+u(F_jF_i^{m-1}p_j)-4u(p_{ij})\big)\\
=\frac{3}{2}\lim_{m\to\infty}5\cdot 3^{m-1}\Big(\big(\frac{5}{3}\big)^{m-1}&\big((u\circ F_i)(F_j^{m-1}p_k)+(u\circ F_i)(F_j^{m-1}p_i)-2(u\circ F_i)(p_{j})\big)\\
+\big(\frac{5}{3}\big)^{m-1}\big((u&\circ F_j)(F_i^{m-1}p_k)+(u\circ F_j)(F_i^{m-1}p_j)-2(u\circ F_j)(p_{i})\big)\Big),
\end{aligned}
\end{equation*}
so the lemma follows directly.
\end{proof}

\begin{lemma}\label{symfuncs}
Let $u\in E(\lambda)\setminus\{0\}$ with level of birth $m_0=1$ and $\lambda_1\notin\{2,5,6\}$. Then,

(a). $u$ is constant on $\overline{p_{23}p_{31}p_{12}}$ if $u(p_1)=u(p_2)=u(p_3)\neq 0$;

(b). $u$ is constant on $\overline{p_ip_j}$ for some $ij\in S_1$ if 
\[u(p_i)=u(p_j)\neq 0,\quad u(p_k)=\Big(1-\frac{\Phi(\lambda_1)}{2}\Big)u(p_i),\]

\noindent where $\overline{p_{23}p_{31}p_{12}}$ denotes the boundary of the triangle with vertices $p_{23},\, p_{31},\, p_{12}$, and $\overline{p_ip_j}$ denotes the line segment with end points $p_i,\,p_j$.
\end{lemma}
\begin{proof} We first prove (b). Assume $u(p_2)=u(p_3)=1$ and $u(p_1)=1-\frac{\Phi(\lambda_1)}{2}$. By  (\ref{eqdeci}), a direct calculation gives
\begin{align*}
&u(p_{23})=\frac{2(4-\lambda_1)+2(1-\frac{\Phi(\lambda_1)}{2})}{(2-\lambda_1)(5-\lambda_1)}=1,\\
&u(p_{12})=u(p_{31})=\frac{(4-\lambda_1)(2-\frac{\Phi(\lambda_1)}{2})+2}{(2-\lambda_1)(5-\lambda_1)}=1-\frac{\lambda_1}{2}.
\end{align*}
Then, using induction we have 
\begin{equation*}
u(F_wp_{23})=1,\quad u(F_wp_{12})=u(F_wp_{31})=1-\frac{\lambda_{m+1}}{2}
\end{equation*}
for each $w\in\{2,3\}^m$ and each $m\geq 0$. Therefore, (b) follows from the continuity of $u$. 

Next, assume $u(p_1)=u(p_2)=u(p_3)=1-\frac{\lambda_1}2$. Still by (\ref{eqdeci}), a calculation gives
\begin{equation*}
u(p_{12})=u(p_{23})=u(p_{31})=(1-\frac{\lambda_1}2)\frac{2(4-\lambda_1)+2}{(2-\lambda_1)(5-\lambda_1)}=1.
\end{equation*}
Then by (b) $u$ takes constant $1$ on $\overline{p_{23}p_{31}p_{12}}$ by looking at each $u\circ F_i$.
\end{proof}

\begin{remark}\label{symfuncsre}
The condition (a) in the previous lemma is equivalent to $\tau(u)=\boldsymbol{\theta}$ and condition (b) is equivalent to $\tau(u)=\boldsymbol{\zeta}_{\Phi(\lambda_1),ij}$ when $\Phi(\lambda_1)\in(0,6)$, since by \eqref{defpi},
\begin{equation*}
\pi\big((1-\frac{\Phi(\lambda_1)}{2},1,1)^{\mathrm t}\big)=\frac{Q^{\mathrm t}(1-\frac{\Phi(\lambda_1)}{2},1,1)^{\mathrm t}}{3-\frac{\Phi(\lambda_1)}{2}}=\big(0,-\frac{\Phi(\lambda_1)}{6-\Phi(\lambda_1)}\big)^{\mathrm t}.
\end{equation*}
\end{remark}

\begin{lemma}\label{gamma-delta}
Let $u\in E(\lambda)\setminus\{0\}$ with level of birth $m_0=1$ and $\lambda_1\notin\{2,5,6\}$. Suppose that there exists an extreme set $A$ of $u$ and $ij\in S_1$ such that $p_{ij}\in A$. Then,
\begin{equation}\label{gamma-delta00}
u(p_i)=u(p_j),\quad\frac{1}{2-\lambda_1}\big(2u(p_k)-(2-\Phi(\lambda_1))u(p_i)\big)\big(u(p_i)-u(p_k)\big)\geq 0.
\end{equation}
\end{lemma}
\begin{proof}
We only need to consider the case $ij=23$ by symmetry. 

First, we prove $u(p_2)=u(p_3)$. Suppose $u$ has an extreme set $A$ and $p_{23}\in A$. We claim that $\big(\mathrm{d}(u\circ F_2)\big)_{p_3}=\big(\mathrm{d}(u\circ F_3)\big)_{p_2}=0$. Indeed, if $\big(\mathrm{d}(u\circ F_2)\big)_{p_3}>0$, then by (\ref{defdu}),  for $m$ large enough, we have $u(p_{23})>u(F_2F_3^mp_2)$ or $u(p_{23})>u(F_2F_3^mp_1)$; but by Lemma \ref{matching}, $\big(\mathrm{d}(u\circ F_3)\big)_{p_2}=-\big(\mathrm{d}(u\circ F_2)\big)_{p_3}<0$ giving $u(p_{23})<u(F_3F_2^m(p_3))$ or $u(p_{23})<u(F_3F_2^m(p_1))$, which contradicts to $p_{23}\in A$.

Let $m\geq 1$ be an integer so that $5^{-m}\lambda<\lambda_1^D$. From the previous claim, it holds that $\big(\mathrm{d}(u\circ F_2F_3^{m-1})\big)_{p_3}=\big(\mathrm{d}(u\circ F_3F_2^{m-1})\big)_{p_2}=0$. Applying Lemma \ref{pro43} to $u\circ F_2F_3^{m-1}$ and $u\circ F_3F_2^{m-1}$, we have
\begin{align}\label{gamma-delta01pre}
(4-\lambda_{m})u(p_{23})&=2(u(F_2F_3^{m-1}p_2)+u(F_2F_3^{m-1}p_1))\\
&=2(u(F_3F_2^{m-1}p_3)+u(F_3F_2^{m-1}p_1)).\notag
\end{align}

If $m=1$, then
\begin{equation}\label{eqstar}
\begin{aligned}
2(u(F_2p_2)+u(F_2p_1))=(4-\lambda_{1})u(p_{23})=2(u(F_3p_3)+u(F_3p_1)),
\end{aligned}
\end{equation}
so $u(p_2)+u(p_{12})=u(p_3)+u(p_{31})$.
Applying (\ref{eqdeci}) to $u(p_{12})$ and $u(p_{31})$, a calculation gives
$u(p_2)=u(p_3)$.

If $m\geq 2$, then
\begin{equation*}
\begin{aligned}
&u(F_2F_3^{m-1}p_2)+u(F_2F_3^{m-1}p_1)\\
=&\ \frac{(6-\lambda_{m})(u(F_2F_3^{m-2}p_2)+u(F_2F_3^{m-2}p_1))+2(4-\lambda_{m})u(p_{23})}{(2-\lambda_{m})(5-\lambda_{m})},
\end{aligned}
\end{equation*}
which together with \eqref{gamma-delta01pre} yields
\begin{equation*}
\begin{aligned}
&2(u(F_2F_3^{m-2}p_2)+u(F_2F_3^{m-2}p_1))\\
=&\ \frac{(2-\lambda_{m})(5-\lambda_{m})-4}{6-\lambda_{m}}\cdot(4-\lambda_{m})u(p_{23})=(4-\lambda_{m-1})u(p_{23}),
\end{aligned}
\end{equation*}
performing \eqref{gamma-delta01pre} with $m$ replaced by $m-1$.
So by induction this gives \eqref{eqstar}, which also yields $u(p_2)=u(p_3)$.

Next, we prove $\frac{1}{2-\lambda_1}\big(2u(p_1)-(2-\Phi(\lambda_1))u(p_2)\big)\big(u(p_2)-u(p_1)\big)\geq 0$.
By \eqref{eqdeci} and $u(p_2)=u(p_3)$,
\begin{equation*}
u(p_{12})=u(p_{31})=\frac{(6-\lambda_1)u(p_2)+(4-\lambda_1)u(p_1)}{(2-\lambda_1)(5-\lambda_1)}=\frac12(4-\lambda_1)u(p_{23})-u(p_2).
\end{equation*}
Then for $m\geq 1$, by induction we get
\begin{equation}\label{gamma-delta02pre}
\begin{aligned}
\frac{1}{\lambda_m}\big(u(p_{23})&-u(F_2F_3^{m-1}p_2)\big)=\frac{1}{\lambda_m}\big(u(p_{23})-u(F_3F_2^{m-1}p_3)\big)\\
&=\frac{1}{\lambda_1}\big(u(p_{23})-u(p_2)\big)=\frac{1}{\Phi(\lambda_1)(2-\lambda_1)}\big(2u(p_1)-(2-\Phi(\lambda_1))u(p_2)\big),\\
\frac{1}{\lambda_m}\big(u(p_{23})&-u(F_2F_3^{m-1}p_1)\big)=\frac{1}{\lambda_m}\big(u(p_{23})-u(F_3F_2^{m-1}p_1)\big)\\
&=\frac{\lambda_1-2}{2\lambda_1} u(p_{23})+\frac{1}{\lambda_1}u(p_2)=\frac{1}{\Phi(\lambda_1)}\big(u(p_2)-u(p_1)\big).
\end{aligned}
\end{equation}
Since $p_{23}\in A$, the product of the above two equalities gives
\begin{equation*}
\frac{1}{2-\lambda_1}\big(2u(p_1)-(2-\Phi(\lambda_1))u(p_2)\big)\big(u(p_2)-u(p_1)\big)\geq 0.
\end{equation*}
\end{proof}

\begin{lemma}\label{lm4.8}
Let $u\in E(\lambda)\setminus\{0\}$ with level of birth $m_0=1$ and $\lambda_1\notin\{2,5,6\}$. Suppose that there exists an extreme set $A$ of $u$ and $ij\in S_1$ such that $p_{ij}\in A$. Then,

(a).  for $0<\lambda<\lambda_1^D$, 
 $\tau(u)\in\mathcal{L}_{\psi^{-1}(\lambda),ij}\cup\{\boldsymbol{\theta}\}$; 

(b). for $5^{-m}\lambda<\lambda_1^D$ with $m\geq 1$, 
\begin{equation*}
\tau(u\circ F_jF_i^{m-1})\in\mathcal{I}_{\lambda_m,i}\cup\{\boldsymbol{\zeta}_{\lambda_m,ki}\},\quad\tau(u\circ F_iF_j^{m-1})\in\mathcal{I}_{\lambda_m,j}\cup\{\boldsymbol{\zeta}_{\lambda_m,jk}\},
\end{equation*}
where $\mathcal{I}_{\lambda_m,i}$ denotes the open line segment joining $\boldsymbol{\zeta}_{\lambda_m,ki}$ to $\boldsymbol{\zeta}_{\lambda_m,ij}$ for distinct $i,j,k\in S$.
\end{lemma}

\begin{proof} Without loss of generality, assume that $ij=23$.

(a). Since $0<\lambda<\lambda_1^D$, $\lambda_1\in(0,2)$. Recall that $\lambda_0=\Phi(\lambda_1)=\psi^{-1}(\lambda)$. By Lemma \ref{gamma-delta}, $u$\ satisfies (\ref{gamma-delta00}).
Then, using (\ref{pixxi}) it is direct to check that
\begin{equation*}
\begin{aligned}
&\tau(u)\in\pi\circ\tau_\lambda^D\Big(\big\{u\in E(\lambda)\setminus\{0\}: u\text{ satisfies (\ref{gamma-delta00})}\big\}\Big)\\
=&\ \big\{(\xi^{(1)},\xi^{(2)})\in\mathbb{R}^2:\xi^{(1)}=0,\ -\frac{\lambda_0}{6-\lambda_0}\leq\xi^{(2)}\leq 0\big\}=\mathcal{L}_{\lambda_0,23}\cup\{\boldsymbol{\theta}\}\cup\{\boldsymbol{\zeta}_{\lambda_0,23}\},
\end{aligned}
\end{equation*}
since
\begin{equation*}
\begin{aligned}
&(0,1,-1)(\mathbf{1},Q)^{-\mathrm{t}}=\frac{1}{3}\big(0,-2\sqrt{3},0\big),\\
&(2,-(2-\lambda_0),0)(\mathbf{1},Q)^{-\mathrm{t}}=\frac{1}{3}\big(\lambda_0,\sqrt{3}(2-\lambda_0),6-\lambda_0\big),\\
&(-1,1,0)(\mathbf{1},Q)^{-\mathrm{t}}=\frac{1}{3}\big(0,-\sqrt{3},-3\big),
\end{aligned}
\end{equation*}
and $\pi\circ \tau_\lambda^D(u')=\frac1{3-\delta}(0,-\delta)^\mathrm{t}\in\mathcal{L}_{\lambda_0,23}$ for a function $u'\in E(\lambda)$ satisfying (\ref{gamma-delta00}) with $u'(p_1)=1-\delta,\ u'(p_2)=u'(p_3)=1$ and $\delta>0$ small enough.

Further, we see that $\tau(u)\neq\boldsymbol{\zeta}_{\lambda_0,23}$, since otherwise by Remark \ref{symfuncsre}, $u$ is constant on $\overline{p_2p_3}$, which contradicts to $p_{23}\in A$. Therefore, we get $\tau(u)\in\mathcal{L}_{\lambda_0,23}\cup\{\boldsymbol{\theta}\}=\mathcal{L}_{\psi^{-1}(\lambda),23}\cup\{\boldsymbol{\theta}\}$.

(b). For $5^{-m}\lambda<\lambda_1^D$ with $m\geq 1$, it is easy to see that 
$\lambda_{m}=\psi^{-1}(5^{-m}\lambda)\in(0,6)$ by Remark \ref{re33}. Considering the function $v=u\circ F_2F_3^{m-1}$, by (\ref{gamma-delta01pre}), (\ref{gamma-delta02pre}) and $p_{23}=F_2F_3^{m-1}p_3\in A$, we see 
\begin{equation}\label{gamma-delta0102r}
\begin{aligned}
&2v(p_1)+2v(p_2)-(4-\lambda_m)v(p_3)=0,\\
&\big(v(p_{3})-v(p_2)\big)\big(v(p_{3})-v(p_1)\big)\geq 0.
\end{aligned}
\end{equation}
Using (\ref{pixxi}) we find that
\begin{equation*}
\begin{aligned}
&\tau(v)\in\pi\circ\tau_\lambda^D\Big(\big\{v\in E(5^{-m}\lambda)\setminus\{0\}: v\text{ satisfies (\ref{gamma-delta0102r})}\big\}\Big)\\
=&\ \big\{(\xi^{(1)},\xi^{(2)})\in\mathbb{R}^2:\xi^{(2)}=\sqrt{3}\xi^{(1)}-\frac{\lambda_m}{6-\lambda_m},\ 0\leq\xi^{(1)}\leq\frac{\sqrt{3}\lambda_m}{2(6-\lambda_m)}\big\}\\
=&\ \mathcal{I}_{\lambda_m,3}\cup\{\boldsymbol{\zeta}_{\lambda_m,23}\}\cup\{\boldsymbol{\zeta}_{\lambda_m,31}\},
\end{aligned}
\end{equation*}
since
\begin{equation*}
\begin{aligned}
&(2,2,-(4-\lambda_m))(\mathbf{1},Q)^{-\mathrm{t}}=\frac{1}{3}\big(\lambda_m,-\sqrt{3}(6-\lambda_m),6-\lambda_m\big),\\
&(0,-1,1)(\mathbf{1},Q)^{-\mathrm{t}}=\frac{1}{3}\big(0,2\sqrt{3},0\big),\\
&(-1,0,1)(\mathbf{1},Q)^{-\mathrm{t}}=\frac{1}{3}\big(0,\sqrt{3},-3\big),
\end{aligned}
\end{equation*}
and $\pi\circ \tau_\lambda^D(v')=\frac{\lambda_m}{4(6-\lambda_m)}(\sqrt{3},-1)^\mathrm{t}\in\mathcal{I}_{\lambda_m,3}$ for a function $v'\in E(5^{-m}\lambda)$ satisfying (\ref{gamma-delta0102r}) with $v'(p_1)=v'(p_2)=4-\lambda_m$ and $v'(p_3)=4$.

Also, we can see $\tau(v)\neq\boldsymbol{\zeta}_{\lambda_m,23}$, since otherwise by \eqref{eqpalpha}, $\tau(u)=\boldsymbol{\zeta}_{\Phi(\lambda_1),23}$, and then by Remark \ref{symfuncsre}, $u$ is constant on $\overline{p_2p_3}$, which contradicts $p_{23}\in A$. Therefore we get $\tau(u\circ F_2F_3^{m-1})\in\mathcal{I}_{\lambda_m,3}\cup\{\boldsymbol{\zeta}_{\lambda_m,31}\}$.
Similarly, $\tau(u\circ F_3F_2^{m-1})\in\mathcal{I}_{\lambda_m,2}\cup\{\boldsymbol{\zeta}_{\lambda_m,12}\}$.
\end{proof}

\begin{lemma}\label{pro_class}
Let $u\in\mathscr D$ and $A$ be an extreme set of $u$. Then, $A$ is of one of the following two types:

(a). $A=\{p\}$ for $p\in\mathcal{SG}\setminus V_*$;

(b). $A\cap (V_*\setminus V_0)\neq\emptyset$.
\end{lemma}
\begin{proof}
This is obvious from the connectivity of $A$.  \end{proof}

The following is the main result in the second step. 

\begin{proposition}\label{lemnece}
Let $u\in E(\lambda)\setminus\{0\}$ with $0<\lambda<\lambda_1^D$, and suppose $u$ has an extreme set $A$. Then,

(a). if $A$ belongs to type (a) in Lemma \ref{pro_class}, there exists $w\in W_*$ and $ij\in S_1$ such that $\tau(u\circ F_w)\in\mathcal{M}_{5^{-|w|}\lambda,ij}\cup\partial\mathcal{M}_{5^{-|w|}\lambda,ij}$;

(b). if $A$ belongs to type (b) in Lemma \ref{pro_class}, there exists $w\in W_*$ and $ij\in S_1$ such that  $\tau(u\circ F_w)\in\mathcal{L}_{\psi^{-1}(5^{-|w|}\lambda),ij}\cup\{\boldsymbol{\theta}\}$.
\end{proposition}
\begin{proof}
(a). Without loss of generality, suppose $u$ attains a local maximum in $A$. Write $A=\{p\}=\bigcap_{m\geq 1}F_{[\omega]_m}\mathcal{SG}$, where $\omega\in\Sigma$ and $[\omega]_m\in W_m$ is the $m$-th truncation of $\omega$. Since $u$ is continuous, we have $\max_{l\in S}u(F_{[\omega]_m}p_l)\to u(p)$ as $m\to\infty$, which gives $\max_{l\in S}u(F_{[\omega]_m}p_l)\leq u(p)$ for large $m$.

Then, at least one of the following two cases happens:

(1). $\max_{l\in S}u(F_{[\omega]_{m+1}}p_l)>\max_{l\in S}u(F_{[\omega]_{m}}p_l)$ for some large $m$, giving $\tau(u\circ F_{[\omega]_{m}})\in \mathcal{M}_{5^{-m}\lambda,ij}$ for some $ij\in S_1$;

(2). $\max_{l\in S}u\big(F_{[\omega]_{m}}(p_l)\big)=u(p)$ for all large $m$.

When case (2) happens, noticing that $p\notin V_*$, there exist distinct $i,j\in S$ and large $m$ so that   $[\omega]_{m+2}=[\omega]_mij$. It follows that $F_{wij}V_0\cap F_wV_0=\emptyset$ and $\max_{l\in S}u(F_{wij}p_l)=\max_{l\in S}u(F_{wi}p_l)=\max_{l\in S}u(F_{w}p_l)=u(p)$ with $w=[\omega]_m$. Then, one of the following cases happens:

(2-1). $\max_{l\in S}u(F_{w}p_l)=u(F_wp_{ij})$, which gives  $\tau(u\circ F_w)\in\partial\mathcal{M}_{5^{-m}\lambda,ij}$;

(2-2). $\max_{l\in S}u(F_{w}p_l)=\max_{l\in S}u(F_{wi}p_l)=u(F_{wi}p_{ij})$ or $u(F_{wi}p_{jk})$ with $k\in S\setminus\{i,j\}$, which gives $\tau(u\circ F_{wi})\in\partial\mathcal{M}_{5^{-m-1}\lambda,ij}\cup\partial\mathcal{M}_{5^{-m-1}\lambda,jk}$.

(b). Choose $m\geq0$ to be the smallest integer so that $A\cap (V_{m+1}\setminus V_m)\neq\emptyset$ and $p\in A\cap (V_{m+1}\setminus V_{m})$.  Then $A\subset F_w(\mathcal{SG}\setminus V_0)$ with some $w\in W_m$. Write $p=F_wp_{ij}$ for some $ij\in S_1$. Applying Lemma \ref{lm4.8}-(a) to $u\circ F_w$, we get $\tau(u\circ F_w)\in\mathcal{L}_{\psi^{-1}(5^{-m}\lambda),ij}\cup\{\boldsymbol{\theta}\}$.
\end{proof}

Now we come to prove Theorem \ref{thm4}.

\begin{proof}[Proof of Theorem \ref{thm4}]
First, we prove $\mathcal{A}_\lambda=\mathcal{C}_\lambda=\mathcal{D}_{\psi^{-1}(\lambda)}$. By Proposition \ref{lem44}, we already have $\mathcal{C}_\lambda=\mathcal{D}_{\psi^{-1}(\lambda)}$ for $0<\lambda<\lambda_1^D$.

Let $u\in E(\lambda)\setminus\{0\}$. If $\tau(u)\in\mathcal{C}_\lambda$, then $(\mathrm{d}u)_{p_1},(\mathrm{d}u)_{p_2}$ and $(\mathrm{d}u)_{p_3}$ are either all positive or all negative. Suppose all positive and $u(p_1)=\min_{i\in S}\{u(p_i)\}$. From $(\mathrm{d}u)_{p_1}=\lim_{m\to\infty}(5/3)^m\big(2u(p_1)-u(F_1^mp_2)-u(F_1^mp_3)\big)>0$ we have $u(F_1^mp_i)<u(p_1)$ for some $i\in\{2,3\}$ and large $m$, which gives $\#\mathrm{Extr}(u)\geq 1$. So $\mathcal{C}_\lambda\subset\mathcal{A}_\lambda$.

If $\tau(u)\notin\mathcal{C}_\lambda\cup\big(\bigcup_{ij\in S_1}\{\boldsymbol{\zeta}_{\lambda_0,ij}\}\big)$, noticing that $\lambda_m=\Phi(\lambda_{m+1})$ and $\lambda_m\in(0,2)$ for each $m\geq 1$, by Lemma \ref{pro41r}-(a), Proposition \ref{lem44}, and induction we have $\tau(u\circ F_w)\notin\mathcal{C}_{5^{-m}\lambda}\cup\big(\bigcup_{ij\in S_1}\{\boldsymbol{\zeta}_{\lambda_m,ij}\}\big)$ for any $m\geq 0$ and $w\in W_m$. Still from Proposition \ref{lem44}, $\mathcal{L}_{\psi^{-1}(5^{-m}\lambda),ij}\subset\mathcal{M}_{5^{-m}\lambda,ij}\subset\mathcal{C}_{5^{-m}\lambda}$ and $\big(\partial\mathcal{M}_{5^{-m}\lambda,ij}\big)\setminus\{\boldsymbol{\zeta}_{\psi^{-1}(5^{-m}\lambda),ij}\}\subset\mathcal{C}_{5^{-m}\lambda}$, we have $\tau(u\circ F_w)\notin\mathcal{M}_{5^{-m}\lambda,ij}\cup\partial\mathcal{M}_{5^{-m}\lambda,ij}$ for any $m\geq 0$, $w\in W_m$ and $ij\in S_1$, which gives $\#\mathrm{Extr}(u)=0$ by Proposition \ref{lemnece}.

If $\tau(u)=\boldsymbol{\zeta}_{\lambda_0,23}$, repeatedly using \eqref{eqpalpha}, we get $\tau(u\circ F_w)=\boldsymbol{\zeta}_{\lambda_m,23}$ for $w\in \{2,3\}^m$; and $\tau(u\circ F_w)\notin\mathcal{C}_{5^{-m}\lambda}\cup \big(\bigcup_{ij\in S_1}\{\boldsymbol{\zeta}_{\lambda_m,ij}\}\big)$ for $w\in W_m\setminus\{2,3\}^m$. So $A\not\subset F_w(\mathcal{SG}\setminus V_0)$ for $w\in W_m\setminus\{2,3\}^m$, and by Lemma \ref{lm4.8}-(a), $A\cap \bigcup_{w\in W_*}\{F_wp_{12},F_wp_{31}\}=\emptyset$. So by Proposition \ref{lemnece}, if there exists an extreme set $A$ of $u$, then $A\cap\overline{p_2p_3}\neq\emptyset$. However, by Remark \ref{symfuncsre}, $u$ is constant on $\overline{p_2p_3}$ in this case. So $\#\mathrm{Extr}(u)=0$. By symmetry, we also have $\#\mathrm{Extr}(u)=0$ when $\tau(u)=\boldsymbol{\zeta}_{\lambda_0,31}$ or $\boldsymbol{\zeta}_{\lambda_0,12}$.

Hence $\#\mathrm{Extr}(u)=0$ for $\tau(u)\notin\mathcal{C}_\lambda$, that is, $\mathcal{A}_\lambda\subset\mathcal{C}_\lambda$. Therefore $\mathcal{A}_\lambda=\mathcal{C}_\lambda=\mathcal{D}_{\psi^{-1}(\lambda)}$.

Then by Proposition \ref{lem44}, we have a disjoint decomposition
\begin{equation}\label{pthm4-01}
\mathcal{A}_{\lambda}=\big(\bigcup_{i\in S}(\mathcal{T}_{\lambda}^i)^{-1}(\mathcal{A}_{5^{-1}\lambda})\big)\cup\big(\bigcup_{ij\in S_1}\mathcal{L}_{\psi^{-1}(\lambda),ij}\big)\cup\{\boldsymbol{\theta}\}
\end{equation}
with $(\mathcal{T}_{\lambda}^i)^{-1}(\mathcal{A}_{5^{-1}\lambda})=\mathcal{G}_{\lambda_0,i}=\mathcal{G}_{\psi^{-1}(\lambda),i}$ for $i\in S$.

Next, we prove
\begin{equation}\label{pthm4-02}
\mathcal{B}_\lambda=\big(\bigcup_{ij\in S_1}\mathcal{L}_{\psi^{-1}(\lambda),ij}\big)\cup\{\boldsymbol{\theta}\}.
\end{equation} 
Obviously, $\mathcal{B}_\lambda\subset\big(\bigcup_{ij\in S_1}\mathcal{L}_{\psi^{-1}(\lambda),ij}\big)\cup\{\boldsymbol{\theta}\}$. It suffice to prove the reverse inclusion.

For $\tau(u)\in\big(\bigcup_{ij\in S_1}\mathcal{L}_{\psi^{-1}(\lambda),ij}\big)\cup\{\boldsymbol{\theta}\}$, (\ref{pthm4-01}) yields that there exists at least one extreme set $A$ of $u$, but $u$ has no extreme set
in $F_i(\mathcal{SG}\setminus V_0)$ for each $i\in S$. So for an extreme set $A$ of $u$, it holds that $A\cap(V_1\setminus V_0)\neq\emptyset$.

If $\tau(u)\in\mathcal{L}_{\psi^{-1}(\lambda),23}$,  for any extreme set $A$ of $u$, Lemma \ref{lm4.8}-(a) gives that $p_{31}\notin A,\ p_{12}\notin A$, so we have $p_{23}\in A$ and $\#\mathrm{Extr}(u)=1$. 
Since $\tau(u)=(\xi^{(1)},\xi^{(2)})^{\mathrm t}\in \mathcal{L}_{\lambda_0,23}$, we have $\xi^{(1)}=\frac{\sqrt{3}(u(p_3)-u(p_2))}{2(u(p_1)+u(p_2)+u(p_3))}=0$ and $-\frac{\lambda_0}{6-\lambda_0}<\xi^{(2)}=\frac{2u(p_1)-u(p_2)-u(p_3)}{2(u(p_1)+u(p_2)+u(p_3))}<0$, which gives $u(p_1)\neq u(p_2)$ and $2u(p_1)\neq (2-\lambda_0)u(p_2)$. Therefore, by (\ref{gamma-delta02pre}), $A=\{p_{23}\}$. By symmetry, $\tau(u)\in\mathcal{L}_{\psi^{-1}(\lambda),31}$ gives $A=\{p_{31}\}$, and $\tau(u)\in\mathcal{L}_{\psi^{-1}(\lambda),12}$ gives $A=\{p_{12}\}$. If $\tau(u)=\boldsymbol{\theta}$,  by Remark \ref{symfuncsre}, $A=\overline{p_{23}p_{31}p_{12}}$. Altogether, for $\tau(u)\in\big(\bigcup_{ij\in S_1}\mathcal{L}_{\psi^{-1}(\lambda),ij}\big)\cup\{\boldsymbol{\theta}\}$, we have $\#\mathrm{Extr}(u)=1$ and $A\cap (V_1\setminus V_0)\neq\emptyset$ for the only extreme set $A$ of $u$. Therefore $\mathcal{B}_\lambda\supset\big(\bigcup_{ij\in S_1}\mathcal{L}_{\psi^{-1}(\lambda),ij}\big)\cup\{\boldsymbol{\theta}\}$ holds.

Finally, condition (A) follows from (\ref{eq2}), (\ref{pthm4-01}) and (\ref{pthm4-02}).
\end{proof}

\section{Proof of Theorem \ref{thm5}}\label{sec5}

Before commencing the proof, we introduce some additional notation.

Adopting the same notation as in Lemma \ref{lm4.8}, denote by $\mathcal{I}_{\alpha,i}$ the open line segment joining $\boldsymbol{\zeta}_{\alpha,ki}$ to $\boldsymbol{\zeta}_{\alpha,ij}$ for distinct $i,j,k\in S$, so we have a disjoint decomposition
\begin{equation*}
\partial\mathcal{D}_\alpha=\big(\bigcup_{i\in S}\mathcal{I}_{\alpha,i}\big)\cup\big(\bigcup_{ij\in S_1}\{\boldsymbol{\zeta}_{\alpha,ij}\}\big).
\end{equation*}



Let $\lambda_0\in(0,6)$ and $\mathbf{a}\in \mathbb R^3\setminus\{\mathbf{0}\}$. For $u^{\eps}$ with $|\eps|=n\geq 0$ (associated with $\lambda_0$ and $\mathbf{a}$)  defined before the statement of Theorem \ref{thm5}, we have $\lambda^{\eps}_{n}=\varphi_{\eps}(\lambda_0)\in(0,6)$ and $\lambda^\eps=5^n\psi(\lambda^{\eps}_{n})$. Note that $5^{-n}\lambda^{\eps}=\psi(\lambda^{\eps}_{n})<\psi(6)=\lambda_1^D$.

For $n\geq 1$ and $\eps\in\{-1,1\}^n$ with $\eps_n=1$, we consider the behavior of $u^\eps$ at points in $V_{n}\setminus V_0$.

For $p\in V_{n}\setminus V_0$, let $0\leq n'<n$ be the unique integer such that $p\in V_{n'+1}\setminus V_{n'}$, and take $w\in W_{n'}$ and $ij\in S_1$ such that $p=F_{w}p_{ij}$. Then we have the following lemma.

\begin{lemma}\label{lmI}
Let $p=F_wp_{ij}\in V_{n}\setminus V_0$ with $w\in W_{n'}$, $0\leq n'<n$. Then there exists an extreme set $A$ of $u^\eps$ such that $p\in A$ if and only if one of the following cases happens.

(a). $\tau(u^\eps\circ F_w)=\boldsymbol{\theta}$;

(b). $\tau(u^\eps\circ F_{wji^{n-n'-1}})\in\mathcal{I}_{\lambda_{n}^{\eps},i}$.

Moreover, when case (a) happens, $A=\overline{p_{23}^wp_{31}^wp_{12}^w}$ and when case (b) happens, $A=\{p\}$.
\end{lemma}
\begin{proof}
By symmetry, we only need to prove the case $i=2, j=3$.

(Necessity). Since $p=F_wp_{23}\in A$, applying Lemma \ref{lm4.8}-(b) to $u^\eps\circ F_{w}$ we obtain $\tau(u^\eps\circ F_{w32^{n-n'-1}})\in\mathcal{I}_{\lambda_{n}^{\eps},2}\cup\{\boldsymbol{\zeta}_{\lambda_{n}^{\eps},12}\}$. When $\tau(u^\eps\circ F_{w32^{n-n'-1}})=\boldsymbol{\zeta}_{\lambda_{n}^{\eps},12}$, by (\ref{eqpalpha}) and induction we see $\tau(u^\eps\circ F_{w3})=\boldsymbol{\zeta}_{\lambda_{n'+1}^{\eps},12}$ and $\tau(u^\eps\circ F_{w})=\boldsymbol{\theta}$.

(Sufficiency)-(a). If $\tau(u^\eps\circ F_{w})=\boldsymbol{\theta}$, by Remark \ref{symfuncsre}, $u^\eps$ is constant on $\overline{p_{23}^wp_{31}^wp_{12}^w}$. Moreover, from (\ref{eqdeci}) we have $u^\eps\circ F_w(F_1F_{w'}p_1)=(1-\frac{\lambda_{n'+m+1}}{2})u^\eps\circ F_w(p_{23})$ for each $w'\in\{2,3\}^m$ and $m\geq 0$. By symmetry, this forces the existance of an extreme set of $u^\eps$ inside $\bigcup_{ij\in S_1}\bigcup_{w'\in\{i,j\}^m,k\in S\setminus\{i,j\}}F_{wkw'}\mathcal{SG}$ for any $m\geq 0$, which implies that $\overline{p_{23}^wp_{31}^wp_{12}^w}$ is an extreme set.

(Sufficiency)-(b). If $\tau(u^\eps\circ F_{w32^{n-n'-1}})\in\mathcal{I}_{\lambda_{n}^{\eps},2}$. By (\ref{eqpalpha}) and induction we see $\tau(u^\eps\circ F_{w3})$ is on the projective line passing through $\boldsymbol{\zeta}_{\lambda_{n'+1}^{\eps},12}$ and $\boldsymbol{\zeta}_{\lambda_{n'+1}^{\eps},23}$, then $\tau(u^\eps\circ F_{w})$ is on the projective line passing through $\boldsymbol{\theta}$ and $\boldsymbol{\zeta}_{\lambda_{n'}^{\eps},23}$. This gives that the function $u^\eps\circ F_{w}$ is symmetric with respect to the reflection of $\mathcal{SG}$ exchanging $p_2$ and $p_3$, so $\tau(u^\eps\circ F_{w23^{n-n'-1}})\in\mathcal{I}_{\lambda_{n}^{\eps},3}$. Recalling that $\mathcal{I}_{\lambda_{n}^{\eps},2}\in\partial\mathcal{D}_{\lambda_n^{\eps}}=\partial\mathcal{C}_{\psi^{-1}(\lambda_n^{\eps})}$ (the equality follows from Proposition \ref{lem44}), we have $(\mathrm{d}(u^\eps\circ F_{w32^{n-n'-1}}))_{p_2}=0$, and $(\mathrm{d}(u^\eps\circ F_{w32^{n-n'-1}}))_{p_1}$, $(\mathrm{d}(u^\eps\circ F_{w32^{n-n'-1}}))_{p_3}$ are either both
positive or both negative. By symmetry now $(\mathrm{d}(u^\eps\circ F_{w32^{n-n'-1}}))_{p_1}$, $(\mathrm{d}(u^\eps\circ F_{w32^{n-n'-1}}))_{p_3}$, $(\mathrm{d}(u^\eps\circ F_{w23^{n-n'-1}}))_{p_1}$ and $(\mathrm{d}(u^\eps\circ F_{w23^{n-n'-1}}))_{p_2}$ are either all positive or all negative, forcing an extreme set of $u^\eps$ inside $F_{w32^{n-n'-1}}\mathcal{SG}\cup F_{w23^{n-n'-1}}\mathcal{SG}$. But by Theorem \ref{thm4}, there is no extreme set inside each of $F_{w32^{n-n'-1}}\mathcal{SG}$, $F_{w23^{n-n'-1}}\mathcal{SG}$, so there is an extreme set $A$ such that $p=F_wp_{23}\in A$. Furthermore, by a similar way as we did in the the second to last paragraph of the proof of Theorem \ref{thm4}, we see $\{p\}=A$.
\end{proof}

For $n\geq 1$ and $\eps\in\{-1,1\}^n$ with $\eps_n=1$, noticing that $\lambda_0\in(0,6)$, we have 
$\lambda_{n}^{\eps}=\varphi_{1}(\lambda_{n-1}^{\eps})\in(3,5)$.

\begin{lemma}\label{lmpuzzle}
If $n=|\eps|\geq 1$, then $\#\mathrm{Extr}(u^{\eps}\circ F_w)\geq 1$ for $w\in W_{n-1}$.
\end{lemma}
\begin{proof}
Write $\alpha=\lambda_{n}^{\eps}\in(3,5)$. Noticing $5^{-n}\lambda^{\eps}<\lambda_1^D$, by Theorem \ref{thm4}, for each $i\in S$, $\#\mathrm{Extr}(u^{\eps}\circ F_{wi})=1$ if and only if $\tau(u^{\eps}\circ F_{wi})\in\mathcal{D}_\alpha$, otherwise $\#\mathrm{Extr}(u^{\eps}\circ F_{wi})=0$. And by Lemma \ref{lmI}, for each $ij\in S_1$, $\{F_wp_{ij}\}$ is an extreme set if and only if $\tau(u^{\eps}\circ F_{wi})\in\mathcal{I}_{\alpha,j}$ (and meanwhile $\tau(u^{\eps}\circ F_{wj})\in\mathcal{I}_{\alpha,i}$). 

Using Lemma \ref{pro41r}-(b) with (\ref{eqpalpha}), considering $(\mathcal{P}_\alpha^i)^{-1}(\mathcal{D}_\alpha)$ for each $i\in S$, and the pair $(\mathcal{P}_\alpha^i)^{-1}(\mathcal{I}_{\alpha,j})$, $(\mathcal{P}_\alpha^j)^{-1}(\mathcal{I}_{\alpha,i})$ for each $ij\in S_1$, we can find that
\begin{equation*}
\mathbb{R}\mathrm{P}^2\setminus\{\boldsymbol{\theta}\}=\Big(\bigcup_{i\in S}(\mathcal{P}_\alpha^i)^{-1}(\mathcal{D}_\alpha)\Big)\cup\Big(\bigcup_{ij\in S_1}\big((\mathcal{P}_\alpha^i)^{-1}(\mathcal{I}_{\alpha,j})\cup (\mathcal{P}_\alpha^j)^{-1}(\mathcal{I}_{\alpha,i}) \big)\Big),
\end{equation*}
giving $\#\mathrm{Extr}(u^{\eps}\circ F_w)\geq 1$ for $\tau(u^{\eps}\circ F_w)\neq\boldsymbol{\theta}$.
For $\tau(u^{\eps}\circ F_w)=\boldsymbol{\theta}$, Lemma \ref{lmI} gives $\#\mathrm{Extr}(u^{\eps}\circ F_w)=1$.

The lemma follows. 
\end{proof}

\begin{proof}[Proof of Theorem \ref{thm5}]
From Lemma \ref{lmpuzzle}, $\#\mathrm{Extr}(u^{\eps}\circ F_w)\geq 1$ for each $w\in W_{n-1}$; and from Theorem \ref{thm4}, $\#\mathrm{Extr}(u^{\eps}\circ F_w)\leq 1$ for each $w\in W_{n}$. So by Lemma \ref{pro1}-(b) we have 
\begin{equation*}
3^{n-1}\leq \#\mathrm{Extr}(u^{\eps})\leq 3^n+\#(V_n\setminus V_0)\leq 4\cdot 3^n.
\end{equation*}
On the other hand, $\lambda^{\eps}=5^{n}\psi(\lambda_{n}^{\eps})$ with $\lambda_{n}^{\eps}\in(3,5)$ gives $5^{n}\psi(3)<\lambda^{\eps}<5^{n}\psi(5)$. This together with the above estimate yields
\begin{equation*}
\frac{1}{3}(\psi(5))^{-d_S/2}(\lambda^{\eps})^{d_S/2}<\#\mathrm{Extr}(u^{\eps})<4(\psi(3))^{-d_S/2}(\lambda^{\eps})^{d_S/2},
\end{equation*}
which is (\ref{equn}).
\end{proof}

\section{Proof of Theorem \ref{thm1}}\label{sec6}

We now proceed to prove Theorem \ref{thm1} using Theorems \ref{thm4} and \ref{thm5}.

For $u_\lambda\in E_D(\lambda)\setminus\{0\}$ or $E_N(\lambda)\setminus\{0\}$  with $\mathrm{supp}\,u_\lambda=\mathcal{SG}$, recalling Propostion \ref{thm32}, let $m=m_0$ for the (D2), (D5), (N5) and (N6') cases; $m=m_0+1$ for the (D6) and (N6) cases.  Note that $m=1$ for the (D2) and (N6') cases; $m\geq 1$ for the (D5) case; $m\geq 2$ for the (N5) and (N6) cases; and $m\geq 3$ for the (D6) case. 

Note that  in the (N0) case, $\lambda=0$ and $u_\lambda$ is a constant function. For other $u_\lambda$, $\lambda_m\in(0,6)$, so for each $w\in W_m$, $u_\lambda\circ F_w$ is in the class of eigenfunctions considered in Theorem 3.5, say $u_\lambda\circ F_w=u^{\eps}$.
Moreover, by the spectral decimation (Proposition \ref{thm31}), $\lambda=\Psi(m,\eps,\beta)=5^{m+|\eps|}\psi\circ\varphi_\eps(\beta)$ with $\beta\in\{2,3,5\}$. Noticing that $\varphi_\eps=\varphi_\varnothing=\mathrm{id}$ for $|\eps|=0$ and $\varphi_\eps(\beta)\in(3,5)$ for $|\eps|\geq1$, we have
\begin{equation}\label{estilam}
5^{m+|\eps|}\psi(2)\leq\lambda\leq 5^{m+|\eps|}\psi(5).
\end{equation}

In the following two lemmas, we estimate $\#\mathrm{Extr}(u_\lambda)$ when $\eps=\varnothing$.

\begin{lemma}\label{lm61}
Let $u_\lambda\in E_D(\lambda)\setminus\{0\}$ or $E_N(\lambda)\setminus\{0\}$ with $\mathrm{supp}\,u_\lambda=\mathcal{SG}$. Suppose $\eps=\varnothing$, then

(a). for $u_\lambda$ in the (D2) case, $1\leq \#\mathrm{Extr}(u_\lambda)\leq 6$;

(b). for $u_\lambda$ in the (D5) or (N5) case, there exists a constant $c_5>1$ such that
\begin{equation}\label{esti5}
c_5^{-1}\lambda^{d_S/2}\leq \#\mathrm{Extr}(u_\lambda)\leq c_5\lambda^{d_S/2};
\end{equation}

(c). for $u_\lambda$ in the (D6') case, $\#\mathrm{Extr}(u_\lambda)=0$, where $u_\lambda$ is a $\lambda_2^N$-Neumann eigenfunction.
\end{lemma}
\begin{proof}
Write $u=u_\lambda$. We consider $u\circ F_{w'}$ for each $w'\in W_{m-1}$. Since $-\Delta_mu|_{V_m}=\lambda_mu|_{V_m}$, by (\ref{defdeltam}) we have
\begin{equation}\label{eqlamm}
\begin{aligned}
(4-\lambda_m)u(p_{ij}^{w'})=u(p_{i}^{w'})+u(p_{j}^{w'})+u(p_{ik}^{w'})+u(p_{jk}^{w'})
\end{aligned}
\end{equation}
for distinct $i,j,k\in S$, where $p_{i}^{w'}=F_{w'}p_{i}$, $p_{ij}^{w'}=F_{w'}p_{ij}$.

(a). For $u$ in the (D2) case, $m=m_0=1$,  $\lambda_1=2$, $\lambda=5\psi(2)$ and $w'=\varnothing$. By the Dirichlet boundary condition, $u(p_1)=u(p_2)=u(p_3)=0$. Solving (\ref{eqlamm}) we have $u(p_{23})=u(p_{31})=u(p_{12})\neq 0$, which gives $\#\mathrm{Extr}(u)\geq 1$. Applying Theorem \ref{thm4} to $u\circ F_i$, we see that $\#\mathrm{Extr}(u\circ F_i)\leq 1$ for each $i\in S$. So by Lemma \ref{pro1}-(b), $1\leq \#\mathrm{Extr}(u)\leq 6$.

(b). For $u$ in the (D5) or (N5) case, $\lambda_m=5$ and $\lambda=5^m\psi(5)$.
Solving (\ref{eqlamm}) gives  $-(u(p_{23}^{w'})+u(p_{31}^{w'})+u(p_{12}^{w'}))=u(p_2^{w'})+u(p_3^{w'})=u(p_1^{w'})+u(p_3^{w'})=u(p_1^{w'})+u(p_2^{w'})$, forcing $u(p_1^{w'})=u(p_2^{w'})=u(p_3^{w'})$. This gives $\#\mathrm{Extr}(u\circ F_{w'})\geq 1$ for each $w'\in W_{m-1}$. On the other hand, Theorem \ref{thm4} gives $\#\mathrm{Extr}(u\circ F_w)\leq 1$ for each $w\in W_m$. So by Lemma \ref{pro1}-(b) and (\ref{estilam}) we have
\begin{equation*}
\frac{1}{3}\psi(5)^{-d_S/2}\lambda^{d_S/2}\leq 3^{m-1}\leq \#\mathrm{Extr}(u)\leq 3^m+3^{m+1}\leq 4\psi(2)^{-d_S/2}\lambda^{d_S/2}.
\end{equation*}

(c). For $u$ in the (N6') case, $m=m_0=1$, $\lambda_1=3$, $\lambda=\lambda_2^N=5\psi(3)$ and $w'=\varnothing$. By the Neumann boundary condition, applying Lemma \ref{pro43} to each $u\circ F_i,\ i\in S$, we have $u(p_i)=2(u(p_{ki})+u(p_{ij})))$ for distinct $i,j,k\in S$, and solving (\ref{eqlamm}) gives $u(p_{ij})=-\frac{1}{2}(u(p_i)+u(p_j)+2u(p_k))$. So we have $\sum_{i\in S}u(p_i)=4\sum_{ij\in S_1}u(p_{ij})=-8\sum_{i\in S}u(p_i)$, implying $\sum_{i\in S}u(p_i)=0$ and $\tau(u)\in L_\infty$, the line at infinity of $\mathbb{R}\mathrm{P}^2$.

Considering $(\mathcal{P}_{\lambda_1}^i)^{-1}(\mathcal{D}_{\lambda_1})$ for each $i\in S$, and the pair $(\mathcal{P}_{\lambda_1}^i)^{-1}\big(\mathcal{I}_{\lambda_1,j}\cup\{\boldsymbol{\zeta}_{\lambda_1,jk}\}\big)$, $(\mathcal{P}_{\lambda_1}^j)^{-1}\big(\mathcal{I}_{{\lambda_1},i}\cup\{\boldsymbol{\zeta}_{\lambda_1,ki}\}\big)$ for distinct $i,j,k\in S$, noticing $\lambda_1=3$, using Lemma \ref{pro41r}-(c) with (\ref{eqpalpha3}), we see that
\begin{equation*}
\begin{aligned}
\mathbb{R}\mathrm{P}^2\setminus L_\infty&=\Big(\bigcup_{i\in S}(\mathcal{P}_{\lambda_1}^i)^{-1}(\mathcal{D}_{\lambda_1})\Big)\cup\\
&\bigg(\bigcup_{ij\in S_1,k\in S\setminus\{i,j\}}\Big((\mathcal{P}_{\lambda_1}^i)^{-1}\big(\mathcal{I}_{\lambda_1,j}\cup\{\boldsymbol{\zeta}_{\lambda_1,jk}\}\big)\cup (\mathcal{P}_{\lambda_1}^j)^{-1}\big(\mathcal{I}_{{\lambda_1},i}\cup\{\boldsymbol{\zeta}_{\lambda_1,ki}\}\big) \Big)\bigg).
\end{aligned}
\end{equation*}
Since $\tau(u)\in L_\infty$, applying Theorem 3.4 to $u\circ F_i$ we have $\#\mathrm{Extr}(u\circ F_i)=0$ for each $i\in S$; and using Lemma \ref{lm4.8}-(b) for $p_{ij}$ gives that $p_{ij}$ is not in an extreme set of $u$ for each $ij\in S_1$. Therefore $\#\mathrm{Extr}(u)=0$.
\end{proof}

\begin{lemma}\label{lm62}
Let $u_\lambda\in E_D(\lambda)\setminus\{0\}$ or $E_N(\lambda)\setminus\{0\}$ with $\mathrm{supp}\,u_\lambda=\mathcal{SG}$. For $u_\lambda$ in the (D6) or (N6) case and $\eps=\varnothing$, we have
\begin{equation}\label{esti6}
c_6^{-1}\lambda^{d_S/2}\leq \#\mathrm{Extr}(u_\lambda)\leq c_6\lambda^{d_S/2}
\end{equation}
for some constant $c_6>1$.
\end{lemma}
\begin{proof}
Write $u=u_\lambda$. For $u$ in the (D6) or (N6) case, we have $\lambda_m=3$, $\lambda_{m-1}=6$ and $\lambda=5^m\psi(3)$. Now we consider $u\circ F_{w'}$ for each $w'\in W_{m-2}$. Since $-\Delta_{m-1}u|_{V_{m-1}}=\lambda_{m-1}u|_{V_{m-1}}$, for distinct $i,j,k\in S$ we have
\begin{equation*}
\begin{aligned}
-2u(p_{ij}^{w'})=u(p_{i}^{w'})+u(p_{j}^{w'})+u(p_{ki}^{w'})+u(p_{jk}^{w'}),
\end{aligned}
\end{equation*}
and solving the above equations gives $u(p_i^{w'})=-u(p_{ki}^{w'})-u(p_{ij}^{w'})$. Without loss of generality, we assume $u(p_{23}^{w'})\leq u(p_{31}^{w'})\leq u(p_{12}^{w'})$, then 

(1). if $u(p_{23}^{w'})\geq 0$, noticing that $u(p_1^{w'})$, $u(p_2^{w'})$ and $u(p_3^{w'})$ are not all zero (nor are $u(p_{23}^{w'})$, $u(p_{31}^{w'})$ and $u(p_{12}^{w'})$), we have $u(p_{12}^{w'})>0$, and $u(p_i^{w'})\leq 0$ for each $i\in S$, giving an extreme set of $u$ inside $F_{w'}\mathcal{SG}$;

(2). if $u(p_{23}^{w'})<0\leq u(p_{31}^{w'})\leq u(p_{12}^{w'})$, we write $a=-u(p_{31}^{w'})/u(p_{23}^{w'})$, and $b=-u(p_{12}^{w'})/u(p_{23}^{w'})$, then $0\leq a\leq b$, $u(p_1^{w'})=-(a+b)\cdot(-u(p_{23}^{w'}))$, $u(p_2^{w'})=(1-b)\cdot(-u(p_{23}^{w'}))$ and $u(p_3^{w'})=(1-a)\cdot(-u(p_{23}^{w'}))$;

(2-1). if $1-b>-1$, then $u(p_{23}^{w'})$ is less than each of $u(p_{2}^{w'})$, $u(p_{3}^{w'})$, $u(p_{12}^{w'})$ and $u(p_{31}^{w'})$, giving an extreme set of $u$ inside $F_{w'2}\mathcal{SG}\cup F_{w'3}\mathcal{SG}$;

(2-2). if $1-b\leq-1$, then $u(p_{12}^{w'})$ is greater than each of $u(p_{1}^{w'})$, $u(p_{2}^{w'})$ and $u(p_{3}^{w'})$, giving an extreme set of $u$ inside $F_{w'}\mathcal{SG}$;

(3). other case follows symmetrically.

\noindent So $\#\mathrm{Extr}(u\circ F_{w'})\geq 1$ for each $w'\in W_{m-2}$. On the other hand, Theorem \ref{thm4} gives $\#\mathrm{Extr}(u\circ F_w)\leq 1$ for each $w\in W_m$. Therefore,
\begin{equation*}
\frac{1}{9}\psi(5)^{-d_S/2}\lambda^{d_S/2}\leq 3^{m-2}\leq \#\mathrm{Extr}(u)\leq 3^m+3^{m+1}\leq 4\psi(2)^{-d_S/2}\lambda^{d_S/2}.
\end{equation*}
\end{proof}

Now we come to the proof of Theorem \ref{thm1}.
\begin{proof}[Proof of Theorem \ref{thm1}]
Let $u=u_\lambda\in E_D(\lambda)\setminus\{0\}$ or $E_N(\lambda)\setminus\{0\}$ with $\mathrm{supp}\,u=\mathcal{SG}$. When $u$ is in the (N0) case, $\lambda=0$ and $\#\mathrm{Extr}(u)=0$ as $u$ 
is a constant function. So (\ref{equnthm1}) follows trivially. 

Now for each $w\in W_m$, we consider $u\circ F_w$, which is in $E(5^{-m}\lambda)\setminus\{0\}$. Since $u\circ F_w$ is in the class of eigenfunctions considered in Theorem \ref{thm5}, $u\circ F_w=u^\eps$ for some $\eps\in\{-1,1\}^n$ with $n\geq 0$ and $\eps_n=1$ if $n\geq 1$. 

When $|\eps|\geq1$, we have
\begin{equation*}
c^{-1} (5^{-m}\lambda)^{d_S/2}\leq \#\mathrm{Extr}(u\circ F_w)\leq c (5^{-m}\lambda)^{d_S/2},
\end{equation*}
where $c$ is the same constant in Theorem \ref{thm5}. Notice that from (\ref{estilam}), $3^{m+1}=(5^{m+1})^{d_S/2}\leq\psi(2)^{-d_S/2}\lambda^{d_S/2}$, then by Lemma \ref{pro1}-(b) and $\#W_m=3^m$, we obtain
\begin{equation*}
c^{-1}\lambda^{d_S/2}\leq \#\mathrm{Extr}(u)\leq (c+\psi(2)^{-d_S/2})\lambda^{d_S/2}.
\end{equation*}

When $\eps=\varnothing$, the desired estimate follows from Lemmas \ref{lm61} and \ref{lm62}.
\end{proof}




\subsection*{Conflicts of interest} The Authors declare that there is no conflict of interest.

\subsection*{Data availability statement.} Data sharing not applicable to this article as no datasets were generated or analysed during the current study.

\section*{Acknowledgement}
We are grateful to Professor Qingsong Gu for carefully examining the manuscript and for his benefical comments.

\end{document}